\documentclass[10pt,reqno]{amsart}
\usepackage{hyperref}\hypersetup{colorlinks=true, citecolor=blue}
\usepackage{graphicx}
\usepackage{xfrac}
\usepackage{esint,amsfonts,amsmath,amssymb,epsfig,
mathrsfs,bm,accents,yfonts,mathtools}
\numberwithin{equation}{section}

\newtheoremstyle{mytheorem}
{}
{}
{\it}
{\parindent}
{\bf}
{.}
{ }
{\thmnumber{#2.~}\thmname{#1}\thmnote{~\rm#3}}

\newtheoremstyle{myremark}
{}
{}
{\rm}
{\parindent}
{\bf}
{.}
{ }
{\thmnumber{#2.~}\thmname{#1}\thmnote{~\rm#3}}

\newtheoremstyle{myparagraph}
{}
{}
{\rm}
{\parindent}
{\bf}
{}
{ }
{\thmnumber{#2.~}\thmname{#1}\thmnote{#3}}

\theoremstyle{mytheorem}

\newtheorem{theorem}[subsection]{Theorem}
\newtheorem{lemma}[subsection]{Lemma}

\newtheorem{proposition}[subsection]{Proposition}

\theoremstyle{myremark}
\newtheorem{remark}[subsection]{Remark}
\newtheorem*{remark*}{Remark}

\theoremstyle{myparagraph}

\newtheorem*{parag*}{}

\makeatletter
\def\@secnumfont{\sc}
\def\section{\@startsection{section}{1}%
\z@{1.5\linespacing\@plus .2\linespacing}{.7\linespacing}%
{\normalfont\sc\centering}}
\makeatother

\makeatletter
\def\ps@headings{\ps@empty
 \def\@evenhead{%
  \setTrue{runhead}%
  \normalfont\footnotesize
  \rlap{\thepage}\hfil
  \def\thanks{\protect\thanks@warning}%
  \leftmark{}{}\hfil}%
 \def\@oddhead{%
  \setTrue{runhead}%
  \normalfont\footnotesize\hfil
  \def\thanks{\protect\thanks@warning}%
  \rightmark{}{}\hfil \llap{\thepage}}%
\let\@mkboth\markboth}
\makeatother
\makeatletter
\renewenvironment{proof}[1][\proofname]{\par
  \pushQED{\qed}%
  \normalfont \topsep6\p@\@plus6\p@\relax
  \trivlist
  \itemindent\normalparindent
  \item[\hskip\labelsep
    \scshape
    #1\@addpunct{.}]\ignorespaces
}{%
  \popQED\endtrivlist\@endpefalse
}
\providecommand{\proofname}{Proof}
\makeatother
\newcommand{\R}{\mathbb{R}}
\newcommand{\Rn}{\mathbb{R}^n}

\newcommand{\eps}{{\varepsilon}}

\newcommand{\dist}{{\textup {dist}}}

\newcommand\supp{{\rm supp}\,}
\newcommand\res{\mathop{\hbox{\vrule height 7pt width .5pt depth 0pt
\vrule height .5pt width 6pt depth 0pt}}\nolimits}

\newcommand{\loc}{{\textup{loc}}}

\newcommand{\cL}{{\mathcal{L}}}
\newcommand{\cH}{{\mathcal{H}}}

\newcommand\N{{\mathbb N}}

\newcommand\bS{{\mathbb S}}
\newcommand\C{{\mathbb C}}

\newcommand{\de}{\partial}

\newcommand{\ph}{\varphi}

\newcommand{\cHn}{{\mathcal H}^{n-1}}

\newcommand{\EEE}{\mathscr{E}}
\newcommand{\GGG}{\mathscr{G}}

\newcommand{\Bp}{B_1^+}
\newcommand{\Bn}{B_1^\prime}
\newcommand{\Bnm}{B_1^{\prime,-}}
\newcommand{\wde}{w^{\eps,\delta}}
\newcommand{\wdee}{w^{\eps,\delta_\eps}}

\begin{document}

	%
\pagestyle{empty}
\pagestyle{myheadings}
\markboth%
{\underline{\centerline{\hfill\footnotesize%
\textsc{M. Focardi \& E. Spadaro}%
\vphantom{,}\hfill}}}%
{\underline{\centerline{\hfill\footnotesize%
\textsc{Epiperimetric inequality for thin obstacles}%
\vphantom{,}\hfill}}}

	%
\thispagestyle{empty}

~\vskip -1.1 cm

	%

\vspace{1.7 cm}

	%
{\Large\sl\centering
An epiperimetric inequality for the thin obstacle problem
\\
}

\vspace{.4 cm}

	%
\centerline{\sc Matteo Focardi \& Emanuele Spadaro}

\vspace{.8 cm}

{\rightskip 1 cm
\leftskip 1 cm
\parindent 0 pt
\footnotesize

	%
{\sc Abstract.}
We prove an epiperimetric inequality for the thin obstacle problem, thus extending the pioneering 
results by Weiss on the classical obstacle problem (\textit{Invent. Math.}, 138 (1999), no.~1, 23--50).
This inequality provides the means to study the rate of converge of the rescaled solutions to their 
limits, as well as the regularity properties of the free boundary.
\par
\medskip\noindent
{\sc Keywords:} Epiperimetric inequality, thin obstacle problem, regularity of the
free boundary.
\par
\medskip\noindent
{\sc MSC (2010): 35R35.}

\par
}

%
%
\section{Introduction}
In this paper we develop a homogeneity improvement approach to the thin obstacle problem following the
pioneering work by Weiss \cite{Weiss} for the classical obstacle problem.
For the sake of simplicity we have restricted ourselves to the simplest case of the
\textit{Signorini problem} consisting in minimizing the Dirichlet energy
among functions with positive traces on a given hyperplane.
Loosely speaking we show that around suitable free boundary points
the solutions quantitatively improve the degree of their homogeneity. This fact implies
a number of consequences for the study of the regularity property of the free boundary itself
as recalled in what follows.

In order to explain the main results of the paper we recall some of the basic known facts
on the thin obstacle problem that are most relevant for our purposes.
We consider the minimizers of the Dirichlet energy
\begin{equation*}
\EEE(u):=\int_{\Bp}|\nabla u|^2dx
\end{equation*}
in the class of admissible functions
\begin{equation}\label{e:admiss}
\mathscr{A}_w:=\left\{u\in H^1(\Bp):\,u\geq 0\,\, \text{ on }\Bn,\,
u=w\, \text{ on }(\partial B_1)^+\right\},
\end{equation}
where for any subset $A\subseteq\Rn$ we shall indicate by $A^+$ 
the set $A\cap\{x\in\Rn:\,x_n>0\}$, $\Bn:=\partial\Bp\cap\{x_n=0\}$.
The function
$w\in H^1(\Bp)$ above prescribes
the boundary conditions on $(\partial B_1)^+$ and satisfies
the obvious compatibility condition $w\geq 0$ on $\Bn$
(in the usual sense of traces).
For the sake of convenience in what follows we shall automatically extend every function 
in $\mathscr{A}_w$ by even symmetry.
For $u\in\mathrm{argmin}_{\mathscr{A}_w}\EEE$ we denote by $\Lambda(u)$ its coincidence set, 
i.e.~the set where the solution touches the obstacle
\[
\Lambda(u) := \big\{(\hat x, 0) \in \Bn\,: u(\hat x, 0) =0 \big\},
\]
and by $\Gamma(u)$ the free boundary, namely the topological boundary
of $\Lambda(u)$ in the relative topology of $\Bn$.

Points in the free boundary of $u$ can be classified according to their frequencies. 
Indeed, Athanasopoulos, Caffarelli and Salsa have established in \cite[Lemma~1]{Athan-Caff-Sal} 
that in every free boundary point $x_0$ Almgren's type frequency function
\[
(0, 1-|x_0|) \ni r\mapsto N^{x_0}(r,u):=
\frac{r\,\int_{B_r(x_0)}|\nabla u|^2\,dx}{\int_{\de B_r(x_0)}u^2\, d\cHn} 
\]
is nondecreasing and has a finite limit as $r\downarrow 0$ satisfying 
$N^{x_0}(0^+,u)\in[\sfrac32,\infty)$. Clearly, $N^{x_0}(r,u)$ is well-defined if 
$u\vert_{\de B_r(x_0)} \not\equiv 0$, otherwise once can prove that actually $u\equiv 0$
in $B_r(x_0)$. 

Following the original works by Weiss \cite{Weiss98,Weiss}, Garofalo and Petrosyan \cite{Garofalo-Petrosyan}
have then introduced a family of monotonicity formulas 
exploiting a parametrized family of \emph{boundary adjusted energies \`a la} Weiss: 
for 
$x_0 \in \Gamma(u)$, $\lambda>0$ and $r\in (0,1-|x_0|)$ 
\[
W_{\lambda}^{x_0} (r, u) := \frac{1}{r^{n+1}} \int_{B_r(x_0)}|\nabla u|^2\,dx - 
\frac{\lambda}{r^{n+2}}\int_{\de B_r(x_0)} u^2\, d\cHn.
\]
\cite[Theorem~1.4.1]{Garofalo-Petrosyan} shows that the boundary adjusted energy $W_{\lambda}^{x_0}$ 
corresponding to $\lambda=N^{x_0}(0^+,u)$ is monotone non-decreasing. More precisely, 
\begin{equation}\label{e:W}
\frac{d}{dr}W_{\lambda}^{x_0} (r, u)  =
\frac{2}{r^{n+2\lambda}} \int_{\de B_r(x_0)} \left(\nabla u \cdot x - \lambda u \right)^2 d\cHn. 
\end{equation}
Note that the right-hand side of \eqref{e:W} measures the distance of $u$ from a $\lambda$-homogeneous
function, and essentially explains why suitable rescalings of $u$ converge to homogeneous functions.

\medskip

In this paper we show that, analogously to the case of the classical obstacle problem as discovered by Weiss 
\cite{Weiss}, there are classes of points $x_0$ of the free boundary of $u$ where the monotonicity of 
$W_{\lambda}^{x_0}$ can be explicitly quantified, meaning that there exist constants $\gamma, r_0, C >0$ 
such that
\begin{equation}\label{e:Wbis}
W_{\lambda}^{x_0} (r, u)  \leq C\, r^{\gamma} \quad\forall\;r \in (0, r_0), 
\end{equation}
thus leading to the above mentioned homogeneity improvement of the solutions.
Rather than explaining the important consequences of \eqref{e:Wbis} (for which we give only a small 
essay in \S~\ref{s:free boundary}), we focus in this paper on the way
\eqref{e:Wbis} is proven, i.e.~by means of what Weiss called
\textit{epiperimetric inequality} in homage to Reifenberg's famous result \cite{Reifenberg}
on minimal surfaces.
In order to give an idea of the topic, let us discuss here the case of lowest frequency 
$\lambda = \sfrac32$: roughly speaking, in this case the epiperimetric inequality asserts that,
for dimensional constants $\kappa$, $\delta>0$ if $c \in H^1(B_1)$ is a $\sfrac32$-homogeneous function 
with positive trace on $\Bn$ that is $\delta$-close to the cone of $\sfrac32$-homogeneous global 
solutions, then there exists a function $v$ with the same boundary values of $c$ such that
\begin{equation}\label{e:epi giocattolo}
W^0_{\sfrac32}(1,v) \leq (1-\kappa)\,W^0_{\sfrac32}(1,c).
\end{equation}
The derivation of \eqref{e:Wbis} from the epiperimetric inequality 
\eqref{e:epi giocattolo} is then done via simple algebraic relations on the 
boundary adjusted energy $W^0_{\sfrac32}$, linking its derivative with
the energy of the $\sfrac32$-homogeneous extension of the boundary values
of the solution. Comparing the energy of the latter with that of the solution
itself leads to a differential inequality finally implying \eqref{e:Wbis} 
(cp.~\S~\ref{s:free boundary} for the details).

The main focus of the present note is however the method of proving \eqref{e:epi giocattolo}.
There are indeed only few examples of problems in geometric analysis where
such kind of inequality has been established, with a number of far-reaching applications and
consequences.
The first instance is the remarkable work by Reifenberg \cite{Reifenberg} on minimal surfaces,
then successfully extended in various directions: by Taylor \cite{Taylor} for what concerns soap-films 
and soap-bubbles minimal surfaces, by White \cite{White} for tangent cones to two-dimensional 
area minimizing integral currents, by Chang \cite{Chang} for the analysis of branch points
in two-dimensional area minimizing integral currents, and by De Lellis and Spadaro 
\cite{DL-Sp-Memo} for two-dimensional multiple valued functions minimizing the generalized
Dirichlet energy.
In all of these instances the proof of the epiperimetric inequality is constructive, i.e.~it is performed
via an explicit computation of the energy of a suitable comparison solution, most of the time allowing
to give an explicit bound on the constant $\kappa$.

On the other hand, the proof given by Weiss for the classical obstacle problem \cite{Weiss} is indirect,
it exploits an infinite-dimensional version of a simple stability argument: namely, if $\phi \in C^2(\R^n)$
satisfies $\nabla \phi(y_0) =0$ and $D^2\phi(y_0)$ is positive definite, then $y_0 \in \R^n$ is
an attractive point for the dynamical system $\dot y = \nabla \phi(y)$.
In regard to this, it is worth mentioning that the infinite-dimensional extension of the quoted stability 
argument is in general subtle. For what concerns the present paper, the energy involved in the obstacle problem 
is not regular, and in addition, its second variation is not positive definite since there are entire 
directions where the functional is constant.
This point of view shares some similarities with the approach by Allard and Almgren \cite{AA} 
and by Simon \cite{Simon} for the analysis of the asymptotic of minimal cones with isolated 
singularities.
\medskip

Our proof of the epiperimetric inequality for the thin obstacle problem is inspired by the 
fundamental paper by Weiss \cite{Weiss}. For instance, following Weiss \cite{Weiss} the 
method of proof is a contradiction argument.
However, rather than faithfully reproducing the whole proof in \cite{Weiss}, we underline 
two variational principles at the heart of it, that are more likely to be generalized to other 
contexts. 
In the contradiction argument we note that the failure of 
the epiperimetric inequality leads to a quasi-minimality condition for a sequence of auxiliary 
functionals related to the second variation of the original energy. 
The goal is then to understand the asymptotic behavior of such new energies. Indeed, 
the minimizers of their $\Gamma$-limits characterize the directions along which the 
epiperimetric inequality may fail. To exclude this occurrence, another variational argument 
leads to an orthogonality condition between the minimizers of the mentioned 
$\Gamma$-limits and a suitable tangent cone to the spaces of blowups, thus giving a contradiction.

We think that this scheme based on two competing variational principles can be applied in many other 
problems. For this reason, in order to make it as transparent as possible, we have detailed the proof of our 
main results in several steps, hoping that this effort could be useful for the reader.
\smallskip

We are able to prove the epiperimetric inequality for free boundary points in the following two classes:
\begin{itemize}
\item[(1)] for points with lowest frequency $\sfrac32$ (cp. Theorem~\ref{t:epithin32});
\item[(2)] for \textit{isolated} points of the free boundary with frequency $2m$, $m\in\N\setminus\{0\}$
(cp. Theorem~\ref{t:epithin2m}).
\end{itemize}
The condition in (2) for the free boundary points being isolated can equivalently be rephrased
in terms of the properties of the blowup functions. As explained in \S~\ref{s:preliminaries},
they correspond to the points where the blowups are everywhere positive except at the origin.
We remark that it is still an open problem (except for dimension $n=2$) to classify all the possible
limiting values of the frequency: apart from the quoted lower bound $N^{x_0}(0^+, u) \geq \sfrac32$ 
for every $x_0 \in \Gamma(u)$, there are examples of free boundary points with limiting frequency 
equal to $\sfrac{(2m+1)}2$ and $2m$ for every $m\in \N\setminus\{0\}$ (in dimension $n=2$ these are 
the unique possible values).

With given an even frequency, the limitation on the set of points analyzed is an outcome of our indirect 
approach. 
As already explained by Weiss for the classical obstacle problem \cite{Weiss}, for capacitary reasons
all the possible cases cannot be covered. This turns out to be evident from the asymptotic analysis 
in Theorems~\ref{t:epithin32} and \ref{t:epithin2m}. The $\Gamma$-convergence (w.r.to the weak 
$H^1$ topology) of a family of functionals subject to unilateral obstacle conditions is under study. 
The limit obstacle condition is lost if it is imposed on a set of dimension less than or equal to $n-2$.
We are not aware of any direct argument to prove the epiperimetric
inequality for the obstacle problem via comparison solutions, unlike the case of minimal surfaces: 
it has to be expected that, if found, it could be used to cover all the remaining cases.
\smallskip

After the completion of this paper, we discovered that our results have a big overlap with those 
contained in a very recent preprint by Garofalo, Petrosyan and Smit Vega Garcia \cite{Garofalo-Petrosyan-Smit}.
In this paper the Authors investigate the regularity of the points with least frequency of the
free boundary of the Signorini problem with variable coefficients as a consequence of the 
epiperimetric inequality, following the energetic approach developed for the classical obstacle problem 
by the Authors of the present note and Gelli in \cite{Focardi-Gelli-Spadaro}.
In particular, in \cite{Garofalo-Petrosyan-Smit} the epiperimetric inequality
in case (1) above is proved and the results on the regularity of the free boundary covers 
the ones in \S~\ref{s:free boundary}.
Despite this, we think that the remaining cases of the epiperimetric inequality in (2) are interesting, 
and furthermore we believe that the variational approach to the epiperimetric inequality we have developed
can be generalized to other contexts and therefore that it is worth being shared with the community.
\smallskip

The paper is organized in the following way.
In \S~\ref{s:preliminaries} we give the necessary preliminaries in order to state
the epiperimetric inequality. Then in \S~\ref{s:epi} we provide the proof of the main
results. We state two versions of the epiperimetric inequality covering
the case (1) and (2) above separately, in Theorem~\ref{t:epithin32} and Theorem~\ref{t:epithin2m}
respectively. The proofs of these two results are divided into a sequence of steps, that 
as explained above are meant to give a clear overview on the structure of the proof. 
Since the two proofs are very much similar, we provide all the details for the first case 
and for what concerns the second one, where actually several simplifications occur, we only point 
out the main differences. Finally in \S~\ref{s:free boundary} we prove the regularity of the free
boundary near points of least frequency as a consequence of the epiperimetric inequality.

%
%
\section*{Acknowledgements}
For this research E.~Spadaro has been partially supported by GNAMPA
Gruppo Nazionale per l'Analisi Matematica, 
la Probabilit\`a e le loro Applicazioni of the Istituto Nazionale di Alta Matematica (INdAM) through 
a Visiting Professor Fellowship. 
E.~Spadaro is very grateful to the DiMaI ``U. Dini'' of the University of Firenze 
for the support during the visiting period.

Part of this work was conceived when M. Focardi was visiting the Max-Planck-Institut in Leipzig
in several occurrences. 
He would like to warmly thank the Institute for providing a very stimulating scientific atmosphere
and for all the support received.

M.~Focardi is member of the Gruppo Nazionale per l'Analisi Matematica, la Probabilit\`a e le loro 
Applicazioni (GNAMPA) of the Istituto Nazionale di Alta Matematica (INdAM). 

%
%
\section{Notation and preliminaries}\label{s:preliminaries}
The open ball in the Euclidean space $\Rn$ centered in $x$
and with radius $r$ is denoted by $B_r(x)$,
and by $B_r$ if the center is the origin.
We recall that, given any subset $A\subseteq\Rn$,
we shall indicate by $A^+$ 
the set $A\cap\{x\in\Rn:\,x_n>0\}$. Moreover we set
\[
\Bn:=\partial\Bp\cap\{x_n=0\} \quad \text{and}\quad 
\Bnm:=\Bn\cap\{x_{n-1}\leq 0\}.
\]
The Euclidean scalar product in $\Rn$ among the vectors $\xi_1$ and
$\xi_2$ shall be denoted by $\xi_1\cdot\xi_2$.
The scalar product in the Sobolev space $H^1(B_1)$ shall be
denoted by $\langle\cdot,\cdot\rangle$, and the corresponding norm
by $\|\cdot\|_{H^1}$.


We introduce a parametrized family of \emph{boundary adjusted energies \`a la} Weiss \cite{Weiss}: namely, 
given $\lambda>0$ for every $u\in H^1(B_1)$, we set
\begin{equation}\label{e:adjenrg}
 \GGG_\lambda(u):=\int_{B_1}|\nabla u|^2dx-\lambda\int_{\partial B_1}u^2d\cHn,
\end{equation}
and note that for all $u_1,\,u_2\in\mathscr{A}_w$ in the admissible class
\ref{e:admiss} it holds
\[
\GGG_\lambda(u_1)-\GGG_\lambda(u_2)=2\big(\EEE(u_1)-\EEE(u_2)\big).
\]
Throughout the whole paper we shall be interested only in range of values
$\lambda\in\{\sfrac 32\}\cup\{2m\}_{m\in\N}$.

\subsection{$\sfrac32$-homogeneous solutions}
Next, let $x=(\hat{x},x_n)\in\Rn$, and adopt the notation 
$\bS^{n-2}=\bS^{n-1}\cap\{x_n=0\}$.
Define for $e\in \bS^{n-2}$ 
\[
h_e(x):=
\left(2\,(\hat x \cdot e) - \sqrt{(\hat x \cdot e)^2 + x_n^2} \right)
\sqrt{\sqrt{(\hat{x} \cdot e)^2+x_n^2}+\hat{x} \cdot e}\,.
\]
It is easy to check that
\[
h_e(x)={\sqrt2}\,\mathrm{Re}\big[(\hat{x} \cdot e+i|x_n|)^{\sfrac32}\big],
\]
where one chooses the determination of the complex root satisfying
$h_e\geq 0$ on $\{x_n=0\}$.
Moreover the following properties of $h_e$ hold true:
\begin{enumerate}
\item $h_e$ is even w.r.t.~$\{x_n=0\}$, i.e.
\[
h_e(\hat x, -x_n) = h_e (\hat x, x_n) \quad \forall\; (\hat x, x_n) \in \R^n;
\]
\item $h_e\geq 0$ on $\{x_n=0\}$ and  $h_e=0$ on $\{x_n=0,\; x \cdot e\leq 0\}$;
\item $h_e$ is harmonic on $B_1^+\cup B_1^-$;
\item $h_e\vert_{\{x_n\geq 0\}}$ is $C^{1,\sfrac12}$,
i.e.~there exists $H_e \in \C^{1, \sfrac12}(\R^n)$ such that
\[
H_e\vert_{\{x_n\geq 0\}} = h_e\vert_{\{x_n\geq 0\}};
\]
\item for $x_n>0$ we have
\begin{align}\label{e:dernorm}
\frac{\partial h_e}{\partial x_n}(\hat{x},x_n)
=&-\frac{x_n}{\sqrt{(\hat{x} \cdot e)^2+x_n^2}}\,\sqrt{\sqrt{(\hat{x} \cdot e)^2 + x_n^2} + \hat{x} \cdot \nu}\notag\\
&+
\frac{2\,(\hat x \cdot e) - \sqrt{(\hat x \cdot e)^2 + x_n^2}}{2\,\sqrt{\sqrt{(\hat{x} \cdot e)^2+x_n^2} + \hat{x} \cdot e}}\,
\frac{x_n}{\sqrt{(\hat{x} \cdot e)^2+x_n^2}}.
\end{align}
In particular,
\begin{align}\label{e:dernorm2}
\frac{\partial h_e}{\partial x_n}(\hat{x},0^+)
&:=\lim_{x_n\downarrow 0}\frac{\partial h_e}{\partial x_n}(\hat{x},x_n)\notag\\
&=
\begin{cases}
-\frac{3}{2}|\hat{x}\cdot e|^{\sfrac{3}{4}}
& \text{ on $\{x_n=0, \;x \cdot e< 0\}$} \\
0 & \text{ on $\{x_n=0,\,x \cdot e\geq 0\}$,}
\end{cases}
\end{align}
and by (2)
\begin{equation}\label{e:hhn}
h_e(\hat x, 0)\,\frac{\partial h_e}{\partial x_n}(\hat{x},0^+)=0\quad\text{ on $\{x_n=0\}$}.
\end{equation}
\end{enumerate}

\medskip

We introduce next the set of blow-ups at the \emph{regular points} of the free boundary:
\[
\mathscr{H}_{\sfrac32}:=\left\{\lambda\, h_e:\,e\in\bS^{n-2}, \; \lambda \in [0,\infty) \right\} \subset H^1_{\loc}(\R^n).
\]
Note that $\mathscr{H}_{\sfrac32}$ is a cone in $H^1_{\loc}(\R^n)$,
the restrictions
\[
\mathscr{H}_{\sfrac32}\vert_{B_1} := \left\{f\vert_{B_1}\,: f\in \mathscr{H}_{\sfrac32} \right\} \subset H^1(B_1)
\]
is a closed set and $\mathscr{H}_{\sfrac32}\setminus \{0\}$
is parametrized by an $(n-1)$-dimensional manifold by the map
\[
\bS^{n-2} \times  (0,+\infty) \ni (e, \lambda) \quad \stackrel{\Phi}{\longmapsto} \quad \lambda\, h_e \in \mathscr{H}_{\sfrac32}\setminus\{0\}.
\]
We can then introduce the tangent space to $\mathscr{H}_{\sfrac32}$ at any
point $\lambda_0 h_{e_0}$ as
\begin{equation}\label{e:tg-1}
T_{\lambda\,h_e}\mathscr{H}_{\sfrac32} := \left\{
d_{(\lambda_0, e_0)}\Phi (\xi, \alpha) :\; \xi\cdot e_n = \xi \cdot e_0 =0,\; \alpha \in \R \right\},
\end{equation}
and notice that 
\begin{equation}\label{e:tg}
T_{\lambda\,h_e}\mathscr{H}_{\sfrac32}=
\left\{\alpha\,h_e + v_{e,\xi}:\;
 \xi\cdot e_n = \xi \cdot e_0 =0,\; \alpha \in \R
\right\},
\end{equation}
where we have set
\begin{equation}\label{e:vnuxi}
v_{e,\xi}(x):=
(\hat{x} \cdot \xi)\,\sqrt{\sqrt{(\hat{x} \cdot e)^2+x^2_n}+\hat{x} \cdot e}\,
\end{equation}
Note moreover that
\[
v_{e,\xi}(x) = 
\sqrt{2}\,(\hat{x} \cdot \xi)\, \mathrm{Re}\big[(\hat{x} \cdot e+i|x_n|)^{\sfrac12}\big],
\]
where the determination of the complex square root is chosen in such a way that 
$v_{e,\xi} \geq 0$ in $\{x_n=0\}$.

\medskip

Let us now highlight some additional properties enjoyed by functions $\psi\in\mathscr{H}_{\sfrac32}$. 
For any given $\varphi\in H^1(B_1)$, a simple integration by parts yields
\begin{align*}
\int_{B_1^+}\nabla \psi\cdot\nabla\varphi\,dx=& \int_{B_1^+}\mathrm{div}(\varphi\nabla \psi)dx\notag\\
=& \int_{(\partial B_1)^+}\varphi\frac{\de \psi}{\de\nu}d\cHn
-\int_{\Bn}\varphi\,\frac{\partial \psi}{\partial x_n}(\hat{x},0^+)d\cHn\notag\\
=& \frac32\int_{(\partial B_1)^+}\varphi\,\psi\,d\cHn
-\int_{\Bn}\varphi\,\frac{\partial \psi}{\partial x_n}(\hat{x},0^+)d\cHn,
\end{align*}
where $\nu = \frac{x}{|x|}$ and
we used that $\psi$ is $\sfrac32$-homogeneous and $\Delta \psi =0$ in $B_1^+$.
Therefore, by the even symmetry of $\psi$ we conclude
\begin{equation}\label{e:int-by-parts}
\int_{B_1}\nabla \psi\cdot\nabla\varphi\,dx=
\frac32\int_{\partial B_1}\varphi\,\psi\,d\cHn
-2\int_{\Bn}\varphi\,\frac{\partial \psi}{\partial x_n}(\hat{x},0^+)d\cHn.
\end{equation}
In particular, \eqref{e:int-by-parts} yields that the first variation of $\GGG_{\sfrac 32}$ at 
$\psi\in\mathscr{H}_{\sfrac32}$ in the direction $\varphi\in H^1(B_1)$, formally defined as
\[
\delta\GGG_{\sfrac 32}(\psi)[\varphi]:=
2\int_{B_1}\nabla \psi\cdot\nabla\varphi\,dx
-3\int_{\partial B_1}\psi\,\varphi\,d\cHn,
\]
satisfies
\begin{equation}\label{e:firstvar}
\delta\GGG_{\sfrac 32}(\psi)[\varphi]
=-4\int_{\Bn}\varphi\,\frac{\partial \psi}{\partial x_n}(\hat{x},0^+)d\cHn.
\end{equation}
Furthermore, by taking into account \eqref{e:hhn} and \eqref{e:int-by-parts} 
applied to $\varphi=\psi$, we get
\begin{equation}\label{e:GGGh}
\GGG_{\sfrac 32}(\psi)=0\qquad \text{for all $\psi\in\mathscr{H}_{\sfrac32}$}.
\end{equation}

\subsection{$2m$-homogeneous solutions}
We introduce the even homogeneous solutions representing
the lowest stratum in the singular part of the free-boundary according to 
\cite[Theorem~1.3.8]{Garofalo-Petrosyan}.

More precisely, given $m\in\N\setminus\{0\}$
consider the closed  convex cone of $H^1(B_1)$
\[
\mathscr{H}_{2m}:=\Big\{\psi:\, \psi\text{ $2m$-homogeneous},\, 
\, \triangle \psi=0\,\text{in}\,\R^n,\, \psi=\psi(\hat{x},|x_n|),\,\psi(\hat{x},0)\geq 0\Big\}.
\]
Note that all the functions in $\mathscr{H}_{2m}$ are actually harmonic polynomials 
on $\R^n$ satisfying
\begin{equation}\label{e:dernorm2m}
 \frac{\partial \psi}{\partial x_n}=0\,\, \text{ on }\Bn\,.
\end{equation}
Given $\psi\in\mathscr{H}_{2m}$ and $\varphi\in H^1(B_1)$, an integration by parts 
leads to 
\begin{equation}\label{e:int-by-parts2m}
\int_{B_1}\nabla \psi\cdot\nabla\varphi\,dx=2m\int_{\partial B_1}\varphi\,\psi\,d\cHn.
\end{equation}
Therefore, setting
\[
\delta\GGG_{2m}(\psi)[\varphi] := 2\int_{B_1}\nabla \psi\cdot\nabla\varphi\,dx - 4m\int_{\partial B_1}\varphi\,\psi\,d\cHn,
\]
it is immediate to check that 
\begin{equation}\label{e:firstvar2m}
\delta\GGG_{2m}(\psi)[\varphi]=0\qquad\text{for all $\psi\in\mathscr{H}_{2m}$ and all $\varphi\in H^1(B_1)$,}
\end{equation}
and
\begin{equation}\label{e:GGG2m}
\GGG_{2m}(\psi)=0\qquad \text{for all $\psi\in\mathscr{H}_{2m}$}.
\end{equation}
We can group the functions in $\mathscr{H}_{2m}$ according to the dimension of their invariant 
subspace as follows: for $\psi\in\mathscr{H}_{2m}$ consider the subspace 
\[
 \Pi_\psi=\{\xi\in\R^{n-1}:\,\psi(\hat{x},0)=\psi(\hat{x}+\xi,0)\,\,\,\forall \hat{x}\in\R^{n-1} \},
\]
and, for $d\in\{0,\ldots, n-2\}$ define
\[
 \mathscr{H}_{2m}^{(d)}:=\left\{\psi\in\mathscr{H}_{2m}:\,\dim\Pi_\psi=d\right\}.
\]
In what follows we shall be dealing only with the lowest stratum $\mathscr{H}_{2m}^{(0)}$ for which
we provide the ensuing alternative characterization.
\begin{proposition}
$\psi\in\mathscr{H}_{2m}^{(0)}$ if and only if $\psi(\hat{x},0)>0$ for every $\hat{x}\neq 0$.
\end{proposition}
\begin{proof}
 First note that by the $2m$-homogeneity the condition $\psi|_{\Bn}\geq 0$ 
 implies that actually $\psi(\cdot,0)$ is even w.r.t. $x_i$, for $i\in\{1,\ldots,n-1\}$.
 
 Suppose then by contradiction that $\psi(\hat{y},0)=0$ for some $\hat{y}\neq 0$, and  
 without loss of generality assume that $\sfrac{\hat{y}}{|\hat{y}|}=e_1$. 
 Hence, $\psi(x_1,0,\ldots,0)=0$ for all $x_1\in\R$,
 and we may find a $2(m-\ell)$-homogeneous polynomial $q$, $1\leq \ell<m$, such that 
 $\psi(x)=x_1^{2\ell}q(x)$ and $q(0,x_2,\ldots,x_n)$ is not identically zero. 
 Computing the Laplacian of $\psi$ we get
 \[
  0=\triangle\psi=2\ell(2\ell-1)\,x_1^{2(\ell-1)}q+4\ell x_1^{2\ell-1}\,\frac{\partial q}{\partial x_1}
  +x_1^{2\ell}\triangle q,
 \]
 and thus for all $x_1\neq 0$ we have
\[
 2\ell(2\ell-1)\,q+4\ell x_1\,\frac{\partial q}{\partial x_1}+x_1^2\triangle q=0.
\]
In turn, by letting $x_1\to 0$ we conclude that $q(0,x_2,\ldots,x_n)\equiv 0$, a contradiction.
 \end{proof}
In view of the latter result, it is easy to check that for every $\psi\in\mathscr{H}_{2m}^{(0)}$ the supporting tangent cone to $\mathscr{H}_{2m}^{(0)}$
is given by
\begin{equation}\label{e:tg2m0}
T_\psi\mathscr{H}_{2m}^{(0)}=\widehat{\mathscr{H}}_{2m}:=\Big\{p:\, p \text{ $2m$-homogeneous},\, \triangle p=0\,\text{in}\,\R^n,\, 
p=p(\hat{x},|x_n|)\Big\},
\end{equation}
therefore in particular $ \widehat{\mathscr{H}}_{2m}\supset\mathscr{H}_{2m}$. 

%
%
\section{The epiperimetric inequality}\label{s:epi}
In this section we establish epiperimetric inequalities \textit{\`a la} Weiss 
for the thin obstacle problem. 
In the rest of the section we agree that a function $c \in H^1(B_1)$ is $\lambda$-homogeneous,
$\lambda\in\{\sfrac32\}\cup\{2m\}_{m\in\N\setminus\{0\}}$,
if there exists $f \in H^1_\loc(\R^n)$ which is $\lambda$-homogeneous and satisfying
$c\vert_{B_1} = f$.
Moreover, to avoid cumbersome notation, we shall use, without any risk of ambiguity,
the symbol $\mathscr{H}_{\lambda}$ also to denote the restrictions
of the blowup maps to the unit ball (which in the previous section we denoted by 
$\mathscr{H}_{\lambda}\vert_{B_1}$). Moreover, given $\mathscr{K}\subset H^1(B_1)$ closed set, we define
\[
\dist_{H^1}(c,\mathscr{K}) := \min \left\{ \| c - \ph\|_{H_1(B_1)},\; \ph \in \mathscr{K}\right\}.
\]

\medskip

We are now ready to state the main results of the paper.
\begin{theorem}\label{t:epithin32}
There exist dimensional constants $\kappa\in(0,1)$ and $\delta>0$ such that
if $c\in H^1(B_1)$ is a $\sfrac32$-homogeneous function with
$c\geq 0$ on $\Bn$ and
\begin{equation}\label{e:dist}
 {\mathrm{dist}_{H^1}}\big(c,\mathscr{H}_{\sfrac32}\big)\leq\delta,
\end{equation}
then
\begin{equation}\label{e:epi}
 \inf_{v\in \mathscr{A}_c}\GGG_{\sfrac32}(v)\leq (1-\kappa)\GGG_{\sfrac32}(c).
\end{equation}
\end{theorem}

An analogous result for the lowest stratum of the singular set holds.
 
\begin{theorem}\label{t:epithin2m}
There exist dimensional constants $\kappa\in(0,1)$ and $\delta>0$ such that
if $c\in H^1(B_1)$ is a $2m$-homogeneous function with $c\geq 0$ on $\Bn$ and
\begin{equation}\label{e:dist2m}
 {\mathrm{dist}_{H^1}}\big(c,\mathscr{H}_{2m}\big)\leq\delta\text{ and }P(c)\in \mathscr{H}_{2m}^{(0)}
\end{equation}
where $P:H^1(B_1)\to\mathscr{H}_{2m}$ is the projection operator, then
\begin{equation}\label{e:epi2m}
 \inf_{v\in \mathscr{A}_c}\GGG_{2m}(v)\leq (1-\kappa)\GGG_{2m}(c).
\end{equation}
\end{theorem}

\subsection{The lowest frequency}

Here we prove the epiperimetric inequality for those points of the free boundary 
with frequency $\sfrac 32$. To simplify the notation in the proof below we shall denote
$\GGG_{\sfrac 32}$ only by $\GGG$.
\begin{proof}[Proof of Theorem~\ref{t:epithin32}]
We argue by contradiction. Therefore we start off assuming the existence of numbers 
$\kappa_j,\,\delta_j\downarrow 0$ and of functions $c_j\in H^1(B_1)$ that are 
$\sfrac32$-homogeneous, $c_j\geq 0$ on $\Bn$ and such that
\begin{equation}\label{e:distc}
{\mathrm{dist}_{H^1}}\big(c_j,\mathscr{H}_{\sfrac32}\big)=\delta_j,
\end{equation}
and
\begin{equation}\label{e:epic}
 (1-\kappa_j)\GGG(c_j) \leq \inf_{v\in \mathscr{A}_{c_j}}\GGG(v).
\end{equation}
In particular, setting $h:=h_{e_{n-1}}$, up to a change of coordinates
depending on $j$, we may assume that there exists $\lambda_j\geq0$ such that
\[
\psi_j:=\lambda_j\,h
\]
is a point of minimum distance of $c_j$ from $\mathscr{H}_{\sfrac32}$, i.e.
\begin{equation}\label{e:distc1}
 \|c_j-\psi_j\|_{H^1}={\mathrm{dist}_{H^1}}\big(c_j,\mathscr{H}_{\sfrac32}\big)=\delta_j\qquad\text{for all $j \in \N$}.
\end{equation}

\medskip

We divide the rest of the proof in some intermediate steps.

\medskip

\noindent{\bf Step 1: Introduction of a family of auxiliary functionals.}
We rewrite inequality \eqref{e:epic} conveniently and 
interpret it as an almost minimality condition
for a sequence of new functionals.

\smallskip
 
For fixed $j$, let $v\in \mathscr{A}_{c_j}$ and use \eqref{e:firstvar}
(applied twice to $\psi_j$ with test functions $\varphi=c_j-\psi_j$
and $\varphi=v-\psi_j$) and \eqref{e:GGGh},
in order to rewrite \eqref{e:epic} in the following form
\begin{multline*}
(1-\kappa_j)\Big(\GGG(c_j)-\GGG(\psi_j)-\delta\GGG(\psi_j)[c_j-\psi_j]
 -4\int_{\Bn}(c_j-\psi_j)\,\frac{\partial \psi_j}{\partial x_n}\,d\cHn\Big)\\
\leq \GGG(v)-\GGG(\psi_j)-\delta\GGG(\psi_j)[v-\psi_j]
-4\int_{\Bn}(v-\psi_j)\,\frac{\partial \psi_j}{\partial x_n}\,d\cHn.
\end{multline*}
Simple algebraic manipulations then lead to
\begin{multline}\label{e:quasi minchia}
(1-\kappa_j)
\Big(\GGG(c_j-\psi_j)-4\int_{\Bn}(c_j-\psi_j)\frac{\partial \psi_j}{\partial x_n}\,d\cHn\Big)\\
\leq \GGG(v-\psi_j)-4\int_{\Bn}(v-\psi_j)\frac{\partial \psi_j}{\partial x_n}\,d\cHn,
\end{multline}
for all $v \in \mathscr{A}_{c_j}$.

\medskip

Next we introduce the following notation. We set 
\begin{equation}\label{e:z_j}
z_j:=\frac{c_j-\psi_j}{\delta_j},
\end{equation}
and, recalling that $\psi_j=\lambda_j h$, we set
\[
\vartheta_j:=\frac{\lambda_j}{\delta_j} 
\]
and 
\begin{equation}\label{e:Bj}
{\mathscr{B}}_{j}:=\left\{z\in z_j+H^1_0(B_1):\,(z+\vartheta_jh)\vert_{\Bn}\geq 0\right\}.
\end{equation}
Then we define the functionals $\GGG_j : L^2(B_1) \to (-\infty, +\infty]$
given by
\begin{equation}\label{e:Gj}
\GGG_j(z):=
\begin{cases}
\displaystyle{\int_{B_1}|\nabla z|^2dx-\frac{3}{2}\int_{\partial B_1}z_j^2\,d\cHn
-4\,\vartheta_j\,\int_{\Bn}z\,\frac{\partial h}{\partial x_n}\,d\cHn} & \\
& \hskip-2cm\text{if $z\in {\mathscr{B}_{j}}$,} \\
+\infty  & \hskip-2cm \textup{otherwise.}
\end{cases}
\end{equation}
Note that the second term in the formula does not depend on $z$ but only on the boundary 
conditions $z_j\vert_{\de B_1}$.

\medskip

Therefore, \eqref{e:quasi minchia} reduces to
\begin{equation}\label{e:epic2}
(1-\kappa_j)\GGG_j(z_j)\leq \GGG_j(z)\quad \text{ for all }z\in {\mathscr{B}}_{j}.
\end{equation}
Moreover, note that by \eqref{e:distc1} and \eqref{e:z_j}
\begin{equation}\label{e:distc3}
 \|z_j\|_{H^1(B_1)}=1.
\end{equation}
This implies that we can extract a subsequence (not relabeled) such that
\begin{itemize}
\item[(a)] $(z_j)_{j \in \N}$ converges weakly in $H^1(B_1)$ to some $z_\infty$;
\item[(b)] the corresponding traces $(z_j\vert_{\Bn})_{j \in \N}$
 converge strongly in $L^2(\Bn)\cup L^2(\partial B_1)$;
\item[(c)] $(\vartheta_j)_{j\in\N}$ has a limit $\vartheta\in[0,\infty]$.
\end{itemize}

\medskip

\noindent{\bf Step 2: First properties of $(\GGG_j)_{j\in\N}$.}
We establish the equi-coercivity and some further properties
of the family of the auxiliary functionals $(\GGG_j)_{j\in\N}$.

\smallskip

Notice that for all $w\in \mathscr{B}_j$, being $w\vert_{\de B_1}=z_j\vert_{\de B_1}$, 
it holds that
\begin{equation}\label{e:coerc}
-\int_{\Bn}w\,\frac{\partial h}{\partial x_n}d\cHn=
\int_{\Bn}
-(w+\vartheta_j\,h)\frac{\partial h}{\partial x_n} d\cHn
+\vartheta_j\,\int_{\Bn}
h\frac{\partial h}{\partial x_n} d\cHn
 \geq0,
\end{equation}
where we used \eqref{e:dernorm2}, \eqref{e:hhn} 
and $(w+\vartheta_j\,h)\vert_{\Bn}\geq0$.
Therefore, we deduce from the very definition \eqref{e:Gj} that
for all $w \in \mathscr{B}_j$
\begin{equation}\label{e:equi-coerc}
\int_{B_1}|\nabla w|^2dx-\frac{3}{2}\int_{\partial B_1}z^2_j\leq\GGG_j(w),
\end{equation}
thus establishing the equi-coercivity of the sequence $(\GGG_j)_{j\in\N}$.

By taking into account \eqref{e:distc3}, if $\vartheta\in[0,+\infty)$ then
\begin{equation}\label{e:limGj}
\liminf_j\GGG_j(z_j) \geq 1-\frac32\int_{\de B_1}z_\infty^2\,d\cHn
 -4\vartheta\,\int_{\Bn}z_\infty\,\frac{\partial h}{\partial x_n}\,d\cHn.
\end{equation}
Instead, if $\vartheta=+\infty$ then \eqref{e:distc3} and \eqref{e:equi-coerc} yield
\[
 \liminf_j\GGG_j(z_j)\geq 1-\frac32\int_{\de B_1}z_\infty^2\,d\cHn.
\]
Hence in all instances, it is not restrictive (up to passing to a further subsequence
which we do not relabel) to assume that $(\GGG_j(z_j))_{j\in\N}$ has a limit in 
$(-\infty,+\infty]$. 
Finally, note that 
\begin{equation}\label{e:energia infinita}
\lim_j\GGG_j(z_j)=+\infty \quad \Longleftrightarrow \quad
\lim_j\vartheta_j\int_{\Bn}z_j\,\frac{\partial h}{\partial x_n}\,d\cHn=-\infty.
\end{equation}

\medskip

\noindent{\bf Step 3: Asymptotic analysis of $(\GGG_j)_{j\in\N}$.}
Here we prove a $\Gamma$-convergence result for the family of energies $\GGG_j$. 
\smallskip

More precisely, we distinguish three cases.

\begin{itemize}
\item[(1)] If $\vartheta \in [0, +\infty)$, then
\[
(z_\infty+\vartheta\,h)\vert_{\Bn}\geq0,
\]
and $\Gamma(L^2(B_1))\hbox{-}\lim_j \GGG_j = \GGG_\infty^{(1)}$, 
where
\[
\GGG_\infty^{(1)}(z):=
\begin{cases}
\displaystyle{
\int_{B_1}|\nabla z|^2dx-\frac{3}{2}\int_{\partial B_1}z_\infty^2\,d\cHn
-4\vartheta\,\int_{\Bn}z\,\frac{\partial h}{\partial x_n}\,d\cHn
} & \\
& \hskip-2cm\text{if $z \in {\mathscr{B}^{(1)}_{\infty}}$,}\\
+\infty  & \hskip-2cm \textup{otherwise,}
\end{cases}
\]
and
\[
{\mathscr{B}_{\infty}^{(1)}} := \Big\{z\in z_\infty+H^1_0(B_1): (z+\vartheta\,h)\vert_{\Bn}\geq0\Big\}.
\]

\item[(2)] If $\vartheta = +\infty$ and $\lim_j\GGG_j(z_j)<+\infty$, then
\[
z_\infty\vert_{\Bnm}=0
\]
recalling that $\Bnm=\Bn\cap \{x_{n-1}\leq0\}$, and $\Gamma(L^2(B_1))\hbox{-}\lim_j \GGG_j=\GGG_\infty^{(2)}$,
where
\[
\GGG_\infty^{(2)}(z):=
\begin{cases}
\displaystyle{
\int_{B_1}|\nabla z|^2dx-\frac{3}{2}\int_{\partial B_1}z_\infty^2\,d\cHn} & 
\text{if $z \in {\mathscr{B}^{(2)}_{\infty}}$,}\\
+\infty  & \textup{otherwise,}
\end{cases}
\]
and 
\[
{\mathscr{B}^{(2)}_{\infty}} := \Big\{z\in z_\infty+H^1_0(B_1): z\vert_{\Bnm}=0\Big\}.
\]

\item[(3)] if $\vartheta = +\infty$ and $\lim_j\GGG_j(z_j)=+\infty$, then 
$\Gamma(L^2(B_1))\hbox{-}\lim_j \GGG_j=\GGG_\infty^{(3)}$, where $\GGG_\infty^{(3)}\equiv+\infty$
on the whole $L^2(B_1)$. 

\end{itemize}

Equality $\GGG_\infty^{(i)} = \Gamma(L^2(B_1))\hbox{-}\lim_j \GGG_j$
for $i=1,2,3$ consists, by definition, 
in showing the following two assertions:
\begin{enumerate}
 \item[(a)] for all $(w_j)_{j\in\N}$ and $w\in L^2(B_1)$ such that $w_j\to w$ in $L^2(B_1)$ we have
 \begin{equation}\label{e:liminf}
 \liminf_j\GGG_j(w_j)\geq\GGG_\infty^{(i)}(w);
 \end{equation}
 \item[(b)] for all $w\in L^2(B_1)$ there is $(w_j)_{j\in\N}$ such that $w_j\to w$ in $L^2(B_1)$ and
 \begin{equation}\label{e:limsup}
 \limsup_j\GGG_j(w_j)\leq\GGG_\infty^{(i)}(w).
\end{equation}
\end{enumerate}

\medskip

\noindent{\emph{Proof of the $\Gamma$-convergence in case (1).}}
For what concerns the $\liminf$ inequality \eqref{e:liminf}, we can assume
without loss of generality (up to passing to a subsequence we do not relabel) that 
\[
 \liminf_j\GGG_j(w_j)=\lim_j\GGG_j(w_j)<\infty.
\]
In view of \eqref{e:equi-coerc}, there exists a subsequence (not relabel)
such that $(w_j)_{j\in\N}$ converges to $w$ weakly in $H^1(B_1)$ and thus the corresponding 
traces strongly in $L^2(\de B_1)\cup L^2(\Bn)$.
This implies that $w+\vartheta h\geq 0$ on $\Bn$ and, in particular,
taking $w_j=z_j$ and $w=z_\infty$, we deduce that $z_\infty \in
\mathscr{B}_\infty^{(1)}$.
\eqref{e:liminf} is then a simple consequence of 
the lower semicontinuity of the Dirichlet energy
under the weak convergence in $H^1$.

\smallskip

For what concerns the $\limsup$ inequality \eqref{e:limsup}, 
we start noticing that it is enough to consider the case $w\in\mathscr{B}_\infty^{(1)}$ with
\begin{equation}\label{e:supporto}
\supp(w-z_\infty)\subseteq B_\rho \quad \text{for some $\rho\in(0,1)$.}
\end{equation}
Indeed, in order to deal with the general case,
consider the functions
\[
w_t (x):=w\big(\sfrac{x}{t}\big)\,\chi_{B_1}\big(\sfrac{x}{t}\big)+z_\infty\big(\sfrac{x}{t}\big)\,\chi_{B_t\setminus\overline{B}_1}\big(\sfrac{x}{t}\big)\quad\text{with $t>1$.}
\]
Clearly, $w_t\in H^1(B_1)$, $\supp(w_t-z_\infty)\subseteq B_{\sfrac 1t}$, and $w_t\to w$ 
in $H^1(B_1)$ for $t\downarrow 1$. Since the upper bound inequality
\eqref{e:limsup} holds 
for each $w_t$, a diagonalization argument provides the conclusion.

Moreover, by a simple contradiction argument
it also suffices to show the following: 
given $w$ as in \eqref{e:supporto}, for every sequence $j_k\uparrow+\infty$
there exist subsequences $j_{k_l}\uparrow+\infty$ and $w_l\to w$ in $L^2(B_1)$
such that
\begin{equation}\label{e:limsup-sub}
\limsup_l\GGG_{j_{k_l}}(w_l)\leq\GGG_\infty^{(i)}(w).
\end{equation}

After these reductions, we first use \eqref{e:distc3} to find a subsequences
(not relabeled) such that 
$(|\nabla z_{j_k}|^2\cL^n\res B_1)_{h\in\N}$ converges weakly$^*$ in the sense of measures to some 
finite Radon measure $\mu$.
Fixed $r\in(\rho,1)$, let $R:=\frac{1+r}{2}$ and let $\varphi\in C^1_c(B_1)$ be a cut-off function 
such that
\[
\varphi\vert_{B_r}\equiv 1, \quad \varphi\vert_{B_1\setminus \overline{B}_R}\equiv 0 \quad\text{and}\quad
\|\nabla\varphi\|_{L^\infty}\leq \frac4{1-r}.
\]
Defining
\[
w_k^r:=\varphi\big(w+(\vartheta-\vartheta_{j_k})h\big)+(1-\varphi)z_{j_k},
\]
we easily infer that $w_k^r\in\mathscr{B}_{j_k}$ since $w\in\mathscr{B}_\infty^{(1)}$, 
$z_{j_k}\in\mathscr{B}_{j_k}$ and 
\[
 w_k^r+\vartheta_{j_k}h=\varphi(w+\vartheta h)+(1-\varphi)(z_{j_k}+\vartheta_{j_k} h).
\]
Moreover, since $\vartheta_{j_k}\to\vartheta\in[0,+\infty)$ we get that 
$w_k^r\to\varphi\, w+(1-\varphi)z_\infty$ in $L^2(B_1)$. Simple calculations then lead to 
\begin{multline}\label{e:stimawj}
 \int_{B_1}|\nabla w_k^r|^2dx\leq\int_{B_r}|\nabla w+(\vartheta-\vartheta_{j_k})\nabla h|^2dx\\
 +\underbrace{\int_{B_R\setminus \overline{B}_r}|\nabla w_{j_k}^r|^2dx}_{=:I_k}
 +\int_{B_1\setminus \overline{B}_R}|\nabla z_{j_k}|^2dx
 \end{multline}
By taking into account that $r>\rho$, we estimate the term $I_k$ above as follows
\begin{align}\label{e:Ij}
I_k\leq & \,2\int_{B_R\setminus \overline{B}_r} |\nabla w+(\vartheta-\vartheta_{j_k})\nabla h|^2dx\notag\\
 &+2\int_{B_R\setminus \overline{B}_r}|\nabla z_{j_k}|^2dx
 +2\int_{B_R\setminus \overline{B}_r}|\nabla\varphi|^2|z_\infty-z_{j_k}+(\vartheta-\vartheta_{j_k})h|^2dx.
\end{align}
Hence, provided $\mu(\de B_r)=0$, from \eqref{e:stimawj} and \eqref{e:Ij} we deduce that
\begin{align}\label{e:limsupr}
\limsup_k\int_{B_1}|\nabla w_k^r|^2dx\leq \,\int_{B_R}|\nabla w|^2dx
+\int_{B_R\setminus \overline{B}_r} |\nabla w|^2dx
+ & 3\,\mu(B_1\setminus \overline{B}_r).
\end{align}
In particular, we may apply the construction above to a sequence $r_l\uparrow 1$
and $R_l:=\frac{1+r_l}{2}$, such that $\mu(\de B_{r_l})=0$ for all $l\in\N$.
A diagonalization argument provides a subsequence $j_{k_l}\uparrow\infty$ 
such that $w_l:=w_{k_l}^{r_l}\to w$ in $L^2(B_1)$ and 
\[
\limsup_l\int_{B_1}|\nabla w_l|^2dx\leq \int_{B_1}|\nabla w|^2dx,
\]
and \eqref{e:limsup-sub} follows at once by considering
the strong convergence of traces of $z_{j_{k_l}}$ in $L^2(\Bn)$.

\medskip

\noindent{\emph{Proof of the $\Gamma$-convergence in case (2).}}
For what concerns the $\liminf$ inequality \eqref{e:liminf},
we assume without loss of generality that
\[
 \liminf_j\GGG_j(w_j)=\lim_j\GGG_j(w_j)<\infty.
\]
Since $w_j\in\mathscr{B}_j$, the stated convergences yield that $w\geq 0$ on $\Bnm$. 
Moreover, \eqref{e:coerc} gives
\begin{align*}
0\leq & -\vartheta_j\int_{\Bn}w_j\frac{\de h}{\de x_n}d\cH^{n-1}\leq
\GGG_j(w_j) + \frac{3}{2}\int_{\partial B_1}z^2_j\,d\cH^{n-1}\\
\leq & \sup_j \left(\GGG_j(w_j) +\frac{3}{2}\int_{\partial B_1}z^2_j\,d\cH^{n-1} \right) < +\infty.
\end{align*}
Therefore the convergence of traces implies 
\[
\int_{\Bn}w\frac{\de h}{\de x_n}d\cH^{n-1}=\lim_j\int_{\Bn}w_j\frac{\de h}{\de x_n}d\cH^{n-1}=0.
\]
By \eqref{e:dernorm2} we deduce that actually $w=0$ on $\Bnm$, i.e.~$w\in\mathscr{B}^{(2)}_\infty$. 
In particular, this holds true for $z_\infty$ by taking into account that $\sup_j\GGG_j(z_j)<+\infty$.
The inequality \eqref{e:liminf} then follows at once.

\smallskip

Let us now deal with the $\limsup$ inequality \eqref{e:limsup}.
Arguing as in case (1), we need only to consider the case of 
$w \in\mathscr{B}_\infty^{(2)}$ such that \eqref{e:supporto} holds
and, for every $j_k\uparrow+\infty$, we need to find
a subsequence $j_{k_l}\uparrow+\infty$ and a sequence $w_l\to w$ in $L^2(B_1)$
such that
\begin{equation}\label{e:limsup-sub2}
\limsup_l\GGG_{j_{k_l}}(w_l)\leq\GGG_\infty^{(i)}(w).
\end{equation}

\smallskip

Introduce the positive Radon measures
\[
\nu_k:=|\nabla z_{j_k}|^2\cL^n\res B_1
-4\vartheta_{j_k}(z_{j_k}+\vartheta_{j_k}h)\frac{\de h}{\de x_n}\cH^{n-1}\res \Bnm.
\]
Note that, for $k$ sufficiently large it follows that
\[
\nu_{k}(B_1) = \GGG_{j_k}(z_{j_k})+\frac32\int_{\de B_1}z_{j_k}^2\,d\cH^{n-1}
\leq  \sup_j\GGG_j(z_j)+2\int_{\de B_1}z_{\infty}^2\,d\cH^{n-1}<\infty.
\]
Thus, $(\nu_k)_{k\in\N}$ is equi-bounded in mass and, up to a subsequence that
we do not relabel, we may assume that $(\nu_k)_{k\in\N}$ converges weakly$^\ast$ 
to a finite positive Radon measure $\nu$.

Next we fix two constants $\eps,\delta>0$ sufficiently small and
we introduce the sets 
$K:=\de B_1\cup B_1^{\prime,-}$ and
$G_{\eps}:=\left\{x\in B_1:\,\mathrm{dist}(x,K) > \eps\right\}$ for every $\eps>0$. 
In order to find the sequence in \eqref{e:limsup-sub2}, we modify $w$ in two different steps. 
First we find $\wde \in \mathscr{B}_\infty^{(2)}$ such that
\begin{gather}
\wde\vert_{G_\eps} \in C^\infty(G_\eps) \label{e:step1-regolarizzazione1}\\
\|w-\wde\|_{H^1}^2:=\int_{B_1}\left(|w-\wde|^2 + |\nabla w - \nabla \wde|^2 \right)dx \leq \delta.
\label{e:step1-regolarizzazione2}
\end{gather}
This modification can be achieved in view of Meyers \& Serrin's approximation result that 
provides a function 
$v \in C^\infty(B_1)$ with $\|v-w\|_{H^1} \leq \delta'$, for some small 
$\delta'$ to be specified in what follows, and having the same trace as $w$ on $\de B_1$.
Then set
\[
\wde := \phi_\eps\,v+(1-\phi_\eps)\, w,
\]
where $\phi_\eps:B_1\to[0,1]$ is a smooth cut-off function such that $\phi_\eps\equiv 1$
on $G_\eps$ and $\phi_\eps\equiv 0$ on $B_1\setminus G_{\sfrac{\eps}{2}}$, and $\|\nabla \phi_\eps\|_{L^\infty}\leq \sfrac4\eps$.
Since $\wde-w = \phi_\eps\,(v-w)$, it follows that
\[
\|\wde-w\|^2_{H^1} \leq \|v-w\|^2_{H^1}\big(\|\phi_\eps\|_{L^\infty}^2 + \|\nabla \phi_\eps\|^2_{L^\infty}\big)
\leq (\delta')^2 \left(1+ \frac{16}{\eps^2}\right) \leq \delta,
\]
for $\delta'$ suitably chosen and depending only on $\delta$ and $\eps$.

Next we consider the Lipschitz functions $\psi_\eps, \chi_\eps:B_1 \to \R$ defined by
\[
\psi_{\eps}(x):=
 \begin{cases}
 1 & B_{1-2\eps} \cr
 1-\frac{1-|x|}{\eps} & B_{1-\eps}\setminus B_{1-2\eps} \cr
 0 & B_1\setminus B_{1-\eps},
 \end{cases}
\]
and 
\[
\chi_{\eps}(x):=\Big(2-\frac 1\eps\mathrm{dist}(x,\{x_n=x_{n-1}=0\})\Big)\wedge 1\vee 0.
\]
Note that actually $\chi_{\eps}(x)=\chi_{\eps}(x_{n-1},x_n)$ with 
\[
\chi_{\eps}(x)=
 \begin{cases}
 1 & {x_n^2+x_{n-1}^2}\leq \eps^2 \cr
 0 & {x_n^2+x_{n-1}^2}\geq 4\eps^2.
 \end{cases}
\]
Set $\varphi_{\eps}:=\psi_\eps\wedge(1-\chi_\eps)$, then $\varphi_{\eps}\in\mathrm{Lip}(\R^n,[0,1])$ with
$\|\nabla \varphi_\eps\|_{L^\infty}\leq\frac1\eps$, and
\begin{equation}\label{e:phizero}
\{\varphi_{\eps}=0\}=\{\psi_{\eps}=0\}\cup\{\chi_\eps= 1\}=(B_1\setminus B_{1-\eps})
 \cup\Big\{\sqrt{x_n^2+x_{n-1}^2}\leq \eps\Big\}
\end{equation}
\begin{equation}\label{e:phiuno}
\{\varphi_{\eps}=1\}=\{\psi_{\eps}=1\}\cap\{\chi_\eps=0\}=B_{1-2\eps}\cap\Big\{\sqrt{x_n^2+x_{n-1}^2}\geq 2\eps\Big\}
\end{equation}
We finally define
\[
  w_k^{\eps,\delta}:=\varphi_{\eps}\,\wde+\big(1-\varphi_{\eps}\big)z_{j_k},
\]
where $j_k$ is the sequence considered for \eqref{e:limsup-sub2}.
First, we show that $w_k^{\eps,\delta} \in \mathscr{B}_{j_k}$ for $k$ sufficiently large.
By construction $w_k^{\eps,\delta}\vert_{\de B_1}=z_{j_k}\vert_{\de B_1}$.
Moreover, we have that
\begin{align*}
w_k^{\eps,\delta}+\vartheta_{j_k}\,h=
\varphi_{\eps}\,(\wde+\vartheta_{j_k}\,h)
+\big(1-\varphi_{\eps}\big)(z_{j_k}+\vartheta_{j_k}h),
\end{align*}
and both the two terms above are positive on $\Bn$. Indeed, for what concerns the latter one, 
it is enough to recall that $z_{j_k} \in \mathscr{B}_{j_k}$ and that $1-\varphi_\eps\geq 0$.
Instead, for the former addend we notice that $(\wde+\vartheta_{j_k}\,h)\vert_{\Bnm} = 0 $ and, 
due to the fact that
\begin{itemize}
\item[(i)]   $\wde$ is smooth in $G_\eps$ by \eqref{e:step1-regolarizzazione1},
\item[(ii)]  $h>c_\eps>0$ on $\Bn \cap G_\eps$ for some positive constant $c_\eps>0$,
\item[(iii)] $\supp\varphi_{\eps} \cap \Bn = \bar G_\eps \cap \Bn$ (cp. \eqref{e:phizero}),
\item[(iv)]  $\vartheta_{j_k} \uparrow +\infty$ 
\end{itemize}
for sufficiently large $k$ it follows that
\[
\varphi_\eps\,(\wde+\vartheta_{j_k}\,h)\geq \varphi_\eps\,(-\|\wde\|_{L^\infty(G_\eps \cap \Bn)}
+\vartheta_{j_k}\,c_\eps) \geq 0\quad\text{on $\Bn\setminus\Bnm$}.
\]

We now compute the distance between $w_k^{\eps,\delta}$ and $w$.
By \eqref{e:step1-regolarizzazione2} we have that
\begin{align}\label{e:L2 norm}
\|w_k^{\eps,\delta} - w\|_{L^2} &  \leq \|\wde-w\|_{L^2(B_1)} + \|w-z_{j_k}\|_{L^2(\{\varphi_\eps<1\})}\notag\\
& \leq \delta +\|z_\infty-z_{j_k}\|_{L^2(B_1)}+\|z_\infty\|_{L^2(\{\varphi_\eps<1\})}
+\|w\|_{L^2(\{\varphi_\eps<1\})}.
\end{align}
Furthermore, straightforward computations just like in \eqref{e:stimawj} and \eqref{e:Ij} give
\begin{align}\label{e:energia1}
\int_{B_1} & |\nabla w_k^{\eps,\delta}|^2dx \notag\\  = & \int_{\{\varphi_\eps=1\}}|\nabla \wde|^2dx
+{\int_{\{0<\varphi_\eps<1\}}|\nabla w_k^{\eps,\delta}|^2dx}
+\int_{\{\varphi_\eps=0\}}|\nabla z_{j_k}|^2dx\notag\\
\leq & \int_{\{\varphi_\eps=1\}}|\nabla \wde|^2dx
+2\int_{\{0<\varphi_\eps<1\}}|\nabla \wde|^2dx+2\int_{\{0<\varphi_\eps<1\}}|\nabla z_{j_k}|^2dx\notag\\
+ & \frac{2}{\eps^2}\int_{\{0<\varphi_\eps<1\}}|\wde-z_{j_k}|^2 +\int_{\{\varphi_\eps=0\}}|\nabla z_{j_k}|^2dx.
\end{align}
Taking into account \eqref{e:step1-regolarizzazione2}, we obtain from \eqref{e:energia1} 
\begin{align}\label{e:energia2}
\int_{B_1}|\nabla w_k^{\eps,\delta}|^2dx \leq
& \int_{B_{1}}|\nabla \wde|^2dx+\int_{\{0<\varphi_\eps<1\}}|\nabla \wde|^2dx 
\notag\\
&+3\int_{\{\varphi_\eps<1\}}|\nabla z_{j_k}|^2dx + 
\frac{4}{\eps^2}\int_{\{0<\varphi_\eps<1\}}|w-z_{j_k}|^2+\frac{4\,\delta}{\eps^2}.
\end{align}
By choosing $\delta_\eps = \eps^4$, and setting $w_k^\eps:=w_k^{\eps,\delta_\eps}$, we 
conclude from \eqref{e:energia2} 
that
\begin{align}
\int_{B_1}|\nabla w_k^\eps|^2dx \leq
& \int_{B_{1}}|\nabla \wdee|^2dx
+\int_{\{0<\varphi_\eps<1\}}|\nabla \wdee|^2dx 
\notag\\
&+3\int_{\{\varphi_\eps<1\}}|\nabla z_{j_k}|^2dx 
+ \frac{4}{\eps^2}\int_{\{0<\varphi_\eps<1\}}|w-z_{j_k}|^2+4\,\eps^2.\label{e:energia3}
\end{align}

Next, in view of \eqref{e:dernorm2}, \eqref{e:hhn} and since $(\wdee+\vartheta_{j_k}\,h)\vert_{\Bnm} = 0$,
the very definition of $w_k^\eps$ gives
\begin{align}
0\leq -4\vartheta_{j_k}\int_{\Bn}w_k^\eps\frac{\de h}{\de x_n}d\cH^{n-1} & =
-4\vartheta_{j_k}\int_{B_1^{\prime,-}}\big(1-\varphi_{\eps}\big)
(z_{j_k}+\vartheta_{j_k}h)\frac{\de h}{\de x_n}\cH^{n-1}\notag\\
& \leq 
{-4\vartheta_{j_k}\int_{B_1^{\prime,-}\cap \{\varphi_\eps<1\}}
(z_{j_k}+\vartheta_{j_k}h)\frac{\de h}{\de x_n}d\cH^{n-1}},\label{e:IIk}
\end{align}
having also used in the last inequality that $z_{j_k}\in\mathscr{B}_{j_k}$.

Therefore, by the very definition of $\nu$ and by collecting \eqref{e:energia3} 
and \eqref{e:IIk} we conclude, provided $\nu(\de\{\varphi_\eps<1\})=0$, that 
\begin{align}\label{e:energia4}
\limsup_k\, & \GGG_{j_k}(w_k^{\eps})  \leq  \int_{B_{1}}|\nabla \wdee|^2dx 
+ \int_{\{0<\varphi_\eps<1\}}|\nabla\wdee|^2dx+3\,\nu(\{\varphi_\eps<1\}) \notag\\
+ & \underbrace{\frac{4}{\eps^2}\int_{B_{1-\eps}\cap\{\eps^2<x_{n-1}^2+x_n^2<4\eps^2\}}
|w-z_{\infty}|^2dx}_{I_\eps:=}+4\,\eps^2-\frac32\int_{\de B_1}z_\infty^2\,d\cH^{n-1}, 
\end{align}
where we have used \eqref{e:supporto}, as $\rho>1-\eps$ for $\eps$ small enough, and the equality
\begin{multline*}
\{0<\varphi_\eps<1\}=\\
\{x\in B_{1-\eps}\setminus B_{1-2\eps}:\,x_{n-1}^2+x_n^2>\eps^2\}\cup
\{x\in B_{1-\eps}:\,\eps^2<x_{n-1}^2+x_n^2<4\eps^2\},
\end{multline*}
that follows from \eqref{e:phizero} and \eqref{e:phiuno}. We next claim that 
\begin{equation}\label{e:Ieps}
\lim_{\eps\downarrow 0}I_\eps=0.
\end{equation}
To this aim we use Fubini's theorem, a scaling argument and a $2$-dimensional Poincar\`e inequality 
(recalling that the trace of $w-z_\infty$ is null on $\Bnm$) to deduce that for some 
positive constant $C$ independent of $\eps$ we have
\begin{multline*}
 I_\eps=\frac{4}{\eps^2}\int_{\{y\in\R^{n-2}:\,|y|\leq 1-\eps\}}\hskip-0.5cm dy
 \int_{\{(s,t)\in\R^2:\,\eps^2<s^2+t^2<4\eps^2\wedge((1-\eps)^2-|y|^2)\}}
 |(w-z_{\infty})(y,s,t)|^2ds\,dt\\
 \leq C\,\int_{B_{1-\eps}\cap \{\eps^2<x_{n-1}^2+x_n^2<4\eps^2\}}|\nabla (w-z_\infty)|^2dx,
\end{multline*}
from which \eqref{e:Ieps} follows at once.

To provide the recovery sequence we perform the construction above for a sequence $\eps_i\downarrow 0$ 
such that $\nu(\de \{\varphi_{\eps_i}<1\})=0$ for all $i\in\N$, with the choice $\delta_i := \eps_i^4$.
In view of \eqref{e:step1-regolarizzazione2}, \eqref{e:phiuno}, \eqref{e:L2 norm}, \eqref{e:energia4} 
and \eqref{e:Ieps} a simple diagonal argument implies the existence of a subsequence $j_{k_i}\uparrow\infty$ 
such that $w_{k_i}^{\eps_i}\to w$ in $L^2(B_1)$ and 
\[
\limsup_i\GGG_{j_{k_i}}(w_{k_i}^{\eps_i})\leq\GGG_\infty^{(2)}(w).
\]

\medskip

\noindent{\emph{Proof of the $\Gamma$-convergence in case (3).}}
The proof of \eqref{e:liminf} and \eqref{e:limsup} in case (3)
is immediate: the former follows, indeed, from \eqref{e:epic2} and the fact that
$\lim_j \GGG_j(z) = +\infty$; while the latter is trivial.

\medskip

\noindent{\bf Step 4: Improving the convergence of $(z_j)_{j\in \N}$ if $\lim_j\GGG_j(z_j)<+\infty$.}
Standing the latter assumption, we show that actually $(z_j)_{j\in\N}$ converges
strongly to $z_\infty$ in $H^1(B_1)$.

\smallskip
  
To this aim, we use some standard results in the theory of $\Gamma$-convergence.
The equi-coercivity of $(\GGG_j)_{j\in\N}$ established in \eqref{e:equi-coerc},
the Poincar\`e inequality and the condition $\|z_j\|^2_{H^1}=1$ in \eqref{e:distc3} 
imply the existence of an absolute minimizer $\zeta_j$ of $\GGG_j$ on $L^2$ with
fixed $i\in\{1,2\}$.
By \cite[Theorem~7.4]{dalmaso}, for $i=1,2$ we have that there exists $\zeta_\infty \in H^1(B_1)$
such that
\begin{gather}
\zeta_j\to \zeta_\infty \quad \text{in }\,L^2(B_1),\label{e:Gamma1}\\ 
\GGG_j(\zeta_j)\to\GGG^{(i)}_\infty(\zeta_\infty),\label{e:Gamma2}\\
\text{and $\zeta_\infty$ is the unique minimizer of $\GGG_\infty^{(i)}$,}\label{e:Gamma3}
\end{gather}
where we have used the strict convexity of $\GGG_\infty^{(i)}$
to deduce the uniqueness of the minimizer of $\GGG_\infty^{(i)}$.
In addition, using the strong convergence of the traces
in $L^2(\de B_1)\cup L^2(\Bn)$ and the estimate
\begin{equation}\label{e:zetajbdd}
\GGG_j(\zeta_j)\leq\GGG_j(z_j)\leq \sup_j\GGG_j(z_j)<+\infty, 
\end{equation}
we infer that 
\[
\int_{B_1}|\nabla \zeta_j|^2dx \to \int_{B_1}|\nabla \zeta_\infty|^2dx, 
\]
in turn implying the strong convergence of $(\zeta_j)_{j\in\N}$ to $\zeta_\infty$ in $H^1(B_1)$.

Next note that by \eqref{e:epic2} and \eqref{e:zetajbdd} 
$z_j$ is an almost minimizer of $\GGG_j$,  in the following sense:
\[
0\leq \GGG_j(z_j)-\GGG_j(\zeta_j)\leq \kappa_j\,\GGG_j(z_j)\leq\kappa_j\cdot \sup_j\GGG_j(z_j).
\]
Hence, by taking into account that $\kappa_j$ is infinitesimal, that $z_j\rightharpoonup z_\infty$ 
weakly in $H^1(B_1)$ and \eqref{e:Gamma2}, Step~3 yields that 
\[
\GGG_\infty(z_\infty)\leq \liminf_j\GGG_j(z_j)=\lim_j\GGG_j(\zeta_j)=\GGG_\infty^{(i)}(\zeta_\infty),
\]
in both cases $i=1,2$.
Therefore, by the uniqueness of the absolute minimizer of $\GGG_\infty^{(i)}$, 
we conclude that $z_\infty=\zeta_\infty$. Arguing as above, the strong 
convergence of $(z_j)_{j\in\N}$ to $z_\infty$ in $H^1(B_1)$ follows.
In particular, note that by \eqref{e:distc3} we infer
\begin{equation}\label{e:zinfty}
\|z_\infty\|_{H^1}=1.
\end{equation}

\medskip

The rest of the proof is devoted to find a contradiction to 
all the instances in Step~3.
We start with the easier cases (1) and (3). Instead, to rule out case (2) 
we shall need to establish more refined properties of the function $z_\infty$.

\medskip

\noindent{\bf Step 5: Case (1) cannot occur.} 
We recall what we have achieved so far about $z_\infty$, namely
\begin{itemize}
\item[(i)] $\|z_\infty\|_{H^1}=1$,
\item[(ii)] $z_\infty$ is $\sfrac32$-homogeneous and even with respect to $x_n=0$,
\item[(iii)] $z_\infty$ is the unique minimizer of $\GGG_\infty^{(1)}$ with respect
to its own boundary conditions,
\item[(iv)] $z_\infty\in\mathscr{B}_\infty^{(1)}$, i.e.~$z_\infty +\vartheta\,h\geq 0$
on $\Bn$.
\end{itemize}

As an easy consequence of the properties above, we show now that 
\[
w_\infty:=z_\infty+\vartheta h
\]
minimizes the Dirichlet energy among
all maps $w$ such that $w \in w_\infty + H_0^1(B_1)$ and 
$w\geq 0$ in $\Bn$ in the sense of traces. In other words,
$w_\infty$ is a solution of the Signorini problem.
To show this claim, for every $z \in \mathscr{B}_\infty^{(1)}$ we set
$w:= z +\vartheta\,h$ and by means of \eqref{e:firstvar} we write
\begin{align*}
\GGG_\infty^{(1)}(z) & =\int_{B_1}|\nabla w|^2dx-\vartheta^2\int_{B_1}|\nabla h|^2dx
-\frac{3}{2}\int_{\partial B_1}z_\infty^2\,d\cHn\notag\\
&\quad -2\vartheta\int_{B_1}\nabla z\cdot\nabla h\,dx
-4\vartheta\int_{\Bn}z\frac{\partial h}{\partial x_n}d\cHn\notag\\
& \stackrel{\eqref{e:firstvar}}{=}
\int_{B_1}|\nabla w|^2dx-\vartheta^2\int_{B_1}|\nabla h|^2dx
-\frac{3}{2}\int_{\partial B_1}z_\infty^2\,d\cHn\notag\\
&\quad-3\vartheta\int_{\partial B_1}z_\infty\,h\,d\cHn.
\end{align*}
Therefore, since $z_\infty$ is the unique minimizer of $\GGG_\infty^{(1)}$
and $w_\infty \geq 0$ on $\Bn$, it follows from the previous computation that
$w_\infty$ is a solution of the Signorini problem.
Using now the $\sfrac32$-homogeneity of $w_\infty$ and the classification
of global solutions of the thin obstacle problem with such homogeneity
in \cite[Theorem~3]{Athan-Caff-Sal}
(see also \cite[Proposition~9.9]{PSU}), we deduce that
\begin{equation*}
w_\infty=\lambda_\infty h_{\nu_\infty} \in \mathscr{H}_{\sfrac32},
 \qquad \text{for some $\lambda_\infty \geq 0$ and $\nu_\infty\in\bS^{n-2}$}.
\end{equation*}
Eventually, in view of \eqref{e:distc1}, we reach the desired contradiction:
by the strong convergence $z_j\to z_\infty$ in $H^1(B_1)$ (cp. Step~4 above) 
and by \eqref{e:z_j}, we deduce that 
\begin{equation}\label{e:celopiccolo}
\frac{c_j}{\delta_j} = \vartheta_j h + z_j
\to  \vartheta h + z_\infty = w_\infty \in\mathscr{H}_{\sfrac32} 
\quad \text{ in $H^1(B_1)$,}
\end{equation}
which implies, for $j$ sufficiently large,
\begin{gather*}
\dist_{H^1}(c_j, \mathscr{H}_{\sfrac32})\leq \|c_j-\delta_j\lambda_\infty h_{\nu_\infty}\|_{H^1(B_1)} 
\stackrel{\eqref{e:celopiccolo}}{=} o(\delta_j)
< \delta_j = \dist_{H^1}(c_j, \mathscr{H}_{\sfrac32}),
\end{gather*}
having used in the last line that 
$\delta_j \lambda_\infty h_{\nu_\infty} \in \mathscr{H}_{\sfrac32}$.

\medskip

\noindent{\bf Step 6: Case (3) cannot occur.} 
The heuristic idea to rule out case (3) is to correct the scaling 
of the energies in order to get a non-trivial $\Gamma$-limit for 
the rescaled functionals.

\smallskip

More in details, we start recalling that by \eqref{e:energia infinita}
if $\lim_j\GGG_j(z_j)=+\infty$, then 
\begin{equation}\label{e:gamma_j}
 \gamma_j:=-4\vartheta_j\int_{\Bn}z_j\frac{\de h}{x_n}\,d\cH^{n-1}\uparrow+\infty.
\end{equation}
Further, the convergence $z_j\to z_\infty$ in $L^2(\Bn)$ and \eqref{e:coerc} yield
\[
 \lim_j\frac{\gamma_j}{\vartheta_j}=-4
 \lim_j\int_{\Bn}z_j\frac{\de h}{x_n}\,d\cH^{n-1}=-4\int_{\Bn}z_\infty\frac{\de h}{x_n}\,d\cH^{n-1} \in [0,+\infty),
\]
so that 
\begin{equation}\label{e:thetagamma}
\vartheta_j\,\gamma_j^{-\sfrac12}
\uparrow+\infty.
\end{equation}
It is then immediate to deduce that the right rescaling of the functionals $\GGG_j$ is 
obtained by dividing by a factor $\gamma_j^{-1}$: namely, for every $z\in \mathscr{B}_j$ 
we consider ${\gamma_j^{-1}}\GGG_j(z)$ and notice that
\begin{equation}\label{e:Gjtilde}
{\gamma_j^{-1}}\GGG_j(z)=\widetilde{\GGG_j}\Big({\gamma_j^{-\sfrac12}} z\Big),
\end{equation}
where the functional $\widetilde{\GGG_j}$ is given by
\begin{equation}\label{e:Gj-tilde}
\widetilde{\GGG_j}(w):=
\begin{cases}
\displaystyle{\int_{B_1}|\nabla w|^2dx-\frac{3}{2}\int_{\partial B_1}w^2\,d\cHn
-4\,\frac{\vartheta_j}{\gamma_j^{\sfrac12}}\,\int_{\Bn}w\,\frac{\partial h}{\partial x_n}\,d\cHn} & \\
& \hskip-2cm\text{if $w\in \widetilde{\mathscr{B}_{j}}$,} \\
+\infty  & \hskip-2cm \textup{otherwise,}
\end{cases}
\end{equation}
where 
\begin{equation}\label{e:Bj-tilde}
\widetilde{\mathscr{B}}_{j}:=\left\{w\in \gamma_j^{-\sfrac12}z_j+H^1_0(B_1):\,\big(w+\vartheta_j\,\gamma_j^{-\sfrac12}\,h\big)\vert_{\Bn}\geq 0\right\}.
\end{equation}

Setting $\widetilde{z_j}:={\gamma_j^{-\sfrac12}}{z_j}$, by \eqref{e:distc3} and
$\gamma_j\uparrow+\infty$ we get $\widetilde{z_j}\to 0$ in $H^1(B_1)$.
In addition, \eqref{e:Gjtilde} and the very definition of $\gamma_j$
in \eqref{e:gamma_j} imply that
\begin{equation}\label{e:Gjzjtilde}
\widetilde{\GGG_j}(\widetilde{z_j})=1+O(\gamma_j^{-1}).
\end{equation}
Furthermore, \eqref{e:epic2} rewrites as
\[
 (1-\kappa_j)\widetilde{\GGG_j}(\widetilde{z_j})\leq \widetilde{\GGG_j}(\widetilde{z})
\qquad\text{for all $\widetilde{z}\in\widetilde{\mathscr{B}_j}$.}
\]
In particular, by taking into account \eqref{e:thetagamma}, $\tilde z_j \to 0$
in $H^1(B_1)$ and \eqref{e:Gjzjtilde}, namely 
$\lim_j\widetilde{\GGG_j}(\widetilde{z_j})<+\infty$, we can argue exactly 
as in case (2) of Step~3 to deduce that 
\[
\Gamma(L^2(B_1))\hbox{-}\lim_j\widetilde{\GGG_j}= \widetilde{\GGG_\infty}
\]
with
\[
\widetilde{\GGG_\infty}(\widetilde{z}):=
\begin{cases}
\displaystyle{\int_{B_1}|\nabla \widetilde{z}|^2dx} & 
\text{if $\widetilde{z} \in \widetilde{{\mathscr{B}_{\infty}}}$,}\\
+\infty  & \textup{otherwise,}
\end{cases}
\]
where $\widetilde{\mathscr{B}_{\infty}} := 
\{\widetilde{z}\in H^1_0(B_1):\, \widetilde{z}\vert_{\Bnm}=0\}$.

By Step~4 and the convergence $\widetilde{z_j}\to 0$ in $H^1(B_1)$, the null 
function turns out to be the unique minimizer of $\widetilde{\GGG_\infty}$ and  
$\lim_j \widetilde{\GGG_j}(\widetilde{z_j})=\widetilde{\GGG_\infty}(0)=0$,
thus leading to a contradiction to \eqref{e:Gjzjtilde}.

\medskip

We are then left with excluding case (2) of Step~3 to end the contradiction argument.
To this aim, as already pointed out, we need to investigate more closely the 
properties of the limit $z_\infty$.

From now on we assume that we are in the setting of case (2) of Step~3: i.e.~$\vartheta=+\infty$
and $\lim_j \GGG_j(z_j) < + \infty$.
\medskip

\noindent{\bf Step 7: An orthogonality condition.}
We exploit the fact that $\psi_j$ is a point of minimal distance of $c_j$ 
from $\mathscr{H}_{\sfrac32}$ to deduce that $z_\infty$ is orthogonal to the tangent 
space $T_h\mathscr{H}_{\sfrac32}$.

\smallskip

We start noticing that $\vartheta=+\infty$ implies that
$\lambda_j >0$ for all $j$ large enough. 
Moreover, by the minimal distance condition \eqref{e:distc1}
we infer that, for all $\nu\in\bS^{n-2}$ and $\lambda\geq 0$,
\[
\|c_j-\psi_j\|_{H^1}\leq\|c_j-\lambda h_\nu\|_{H^1},
\]
that, by the very definition of $z_j$ in \eqref{e:z_j}, can actually be rewritten as
\begin{equation*}
\delta_j\|z_j\|_{H^1}\leq\|\psi_j-\lambda h_\nu+\delta_jz_j\|_{H^1},
\end{equation*}
or, equivalently,
\begin{equation}\label{e:orthog}
-\|\psi_j-\lambda h_\nu\|^2_{H^1(B_1)}\leq 2\delta_j
\langle z_j, \psi_j-\lambda h_\nu\rangle.
\end{equation}
Therefore, assuming $(\lambda_j, e_{n-1}) \neq (\lambda, \nu)$ and renormalizing
\eqref{e:orthog}, we get 
\[
-\|\psi_j-\lambda h_\nu\|_{H^1}\leq 2\delta_j
\langle z_j, \frac{\psi_j-\lambda h_\nu}{\|\psi_j-\lambda h_\nu\|_{H^1}}\rangle,
\]
and by taking the limit $(\lambda, \nu) \to (\lambda_j, e_{n-1})$ we conclude
(recall the definition of the tangent space in \eqref{e:tg-1})
\[
\langle z_j, \zeta\rangle=0\qquad\textrm{ for all 
$\zeta\in T_{\psi_j}\mathscr{H}_{\sfrac32}=T_{h}\mathscr{H}_{\sfrac32}$}, 
\]
where we used that $\lambda_j>0$ in computing the tangent vectors.
Now letting $j\uparrow\infty$ in the equality above we get that
\begin{equation}\label{e:perp1}
\langle z_\infty, \zeta\rangle=0\qquad\textrm{ for all $\zeta\in T_{h}\mathscr{H}_{\sfrac32}$}.
\end{equation}

\medskip

\noindent{\bf Step 8: Identification of $z_\infty$ in case (2) of Step~3.}
We show that  
\begin{equation}\label{e:rigidity}
z_\infty (x) = a_0 \,h (x)+ \left(\sum_{i=1}^{n-2} a_i\, x_i\right) \sqrt{\sqrt{x_{n-1}^2+x_n^2}+x_{n-1}},
\end{equation}
for some $a_0, \ldots, a_{n-2} \in \R$, 
i.e.~$z_\infty\in T_h\mathscr{H}_{\sfrac32}$ (cp.~\eqref{e:tg}).

\smallskip

The above claim is consequence of the following facts:
\begin{itemize}
\item[(a)] $z_\infty$ solves the boundary value problem
\begin{gather}\label{e:EL}
\begin{cases}
\Delta z_\infty = 0 &\text{in } B_1\setminus \Bnm,\\
z_\infty =0 & \text{on } \Bnm;
\end{cases}
\end{gather}

\item[(b)] $z_\infty(x',x_n) = z_\infty(x',-x_n)$ for every $(x',x_n) \in B_1$;

\item[(c)] $z_\infty$ is $\sfrac32$-homogeneous.
\end{itemize}

The proof consists of three parts:
\begin{itemize}
\item[(I)] to show the H\"older regularity of $z_\infty$ and of all its transversal derivatives in the sense of distributions
\[
v_{\alpha}:=\frac{\de^{\alpha_1}}{\de x_1^{\alpha_1}}\cdots \frac{\de^{\alpha_{n-2}} z_\infty}{\de x_{n-2}^{\alpha_{n-2}}}  \quad \text{with } \alpha = (\alpha_1, \ldots, \alpha_{n-2}) \in \N^{n-2};
\]

\item[(II)] the use of a bidimensional conformal transformation in the variable $(x_{n-1},x_n)$
to reduce the problem to the upper half ball $B_1^+$;

\item[(III)] the classification of all $\sfrac32$-homogeneous solutions.
\end{itemize}

\smallskip

As for (I), we start noticing that for every $\alpha \in \N^{n-2}$
(in particular also for $z_\infty = v_{(0,\ldots,0)}$),
it follows from (a) and (b) that $v_\alpha\vert_{B_1^+}$
is a solution to the boundary value problem
\begin{gather}\label{e:EL2}
\begin{cases}
\Delta v_\alpha = 0 &\text{in } B_1^+,\\
v_\alpha =0 & \text{on } \Bnm,\\
\frac{\de v_\alpha}{\de x_n} =0 & \text{on } B_1'\setminus \Bnm.
\end{cases}
\end{gather}
It then follows from \cite[Theorem 14.5]{stampacchia} that 
the distributions $v_\alpha$ are represented by
uniformly H\"older continuous functions in every $\bar B_r^+$ with $r<1$.
In particular, by homogeneity,
we conclude that $v_\alpha$ is H\"older continuous
in the whole $B_1\subset \R^n$.

\smallskip

Next, for (II) we follow a suggestion by S.~Luckhaus \cite{luckhaus}
(see also the appendix of \cite{desilva-savin} for a similar procedure)
and consider the conformal transformation $\Phi: B_1^+ \to B_1\setminus \Bnm$ 
defined by 
\begin{equation}\label{e:change}
(x_1, \ldots, x_{n-2}, y_1, y_2) \mapsto (x_1, \ldots, x_{n-2}, y_2^2 - y_1^2, - 2\, y_1\, y_2),
\end{equation}
and set
\[
u_\alpha := v_\alpha \circ \Phi \quad \forall \; \alpha \in \N^{n-2}.
\]
We next introduce the following Laplace operators:
\begin{gather*}
\Delta' := \frac{\de^2}{\de x_1^2} + \cdots + \frac{\de^2}{\de x_{n-2}^2}\\
\Delta'' := \frac{\de^2}{\de x_{n-1}^2} + \frac{\de^2}{\de x_{n}^2}\\
\Delta_y := \frac{\de^2}{\de y_1^2} + \frac{\de^2}{\de y_{2}^2}.
\end{gather*}
By a simple computation, it follows that (we set for simplicity
$x' := (x_1, \ldots, x_{n-2})$)
\begin{align}\label{e:equazione in y}
\Delta_y u_\alpha (x', y_1, y_2) & = 4\,(y_1^2 + y_2^2)\, \Delta'' v_\alpha
(x', y_2^2 - y_1^2, - 2\, y_1\, y_2)\notag\\
& \stackrel{\eqref{e:EL}}{=} - 4\,(y_1^2 + y_2^2)\, \Delta' v_\alpha
(x', y_2^2 - y_1^2, - 2\, y_1\, y_2)
\end{align}
for all $\alpha \in \N^{n-2}$ and for all $(x',y_1,y_2) \in B_1^+$.

Note that the right hand side of \eqref{e:equazione in y} is H\"older
continuous, because by (I) $\Delta' v_\alpha$ is H\"older continuous
for every $\alpha \in \N^{n-2}$.
Therefore, the usual Schauder theory for the Laplace equation implies
that $u_\alpha$ is twice
continuously differentiable with H\"older continuous second order
partial derivatives.

We can then bootstrap this conclusion and infer that
\begin{align}\label{e:equazione in y2}
\Delta' v_\alpha (x', y_2^2 - y_1^2, - 2\, y_1\, y_2) & =
\Delta' u_\alpha (x', y_1, y_2)\notag\\
&= \sum_{i=1}^{n-2} u_{\alpha + 2\,e_i} (x', y_1, y_2)
\end{align}
with $e_1 = (1, 0, \ldots, 0)$, $e_2 = (0, 1, \ldots, 0)$, {\ldots},
$e_{n-2} = (0, 0, \ldots, 1) \in \N^{n-2}$, and therefore
$\Delta'' v_\alpha$ is twice differentiable,
thus implying by \eqref{e:equazione in y} that $u_\alpha$ is $C^{4,\kappa}$
for some $\kappa \in (0,1)$, and so on.
In conclusion, it follows from \eqref{e:EL2},
\eqref{e:equazione in y} and \eqref{e:equazione in y2} that
$u_0 \in C^\infty (\bar B_1^+)$ and
\begin{gather}\label{e:u_0}
u_0(x',y_1, 0) = 0 \quad \forall \; |(x',y_1,0)| <1.
\end{gather}

\smallskip

In order to perform the final classification in (III), we consider a Taylor
expansion of $u_0$ up to order three. For the sake of simplicity
we write $(w_1, \ldots, w_n) = (x', y_1, y_2)$ and use \eqref{e:u_0}
to simplify the expansion: there exist
real numbers $b_l, b_{i,j} \in \R$ for $l \in \{0,\ldots, n\}$
and $i,j \in \{1,\ldots,n\}$ such that $b_{i,j}=b_{j,i}$
for every $i,j$ and 
\[
u_0(w) = w_{n} \left(b_0 + \sum_{i=1}^{n}b_i\, w_i
+ \sum_{i,j=1}^{n} b_{i,j} w_i\,w_j \right) + g(w),
\]
with $g(w) \leq C\, |w|^4$ for every $w \in B_1^+$
for some $C>0$.

Next we perform the change of coordinates $\Phi^{-1}$ to deduce
an expansion for $z_\infty$.
First note that
\begin{equation*}
\Phi^{-1}: B_1\setminus (B_1'\cap\{x_{n-1} \leq 0\})\to B_1^+
\end{equation*}
\begin{equation*}
\Phi^{-1}(x_1, \ldots, x_{n}) =
\big(x_1, \ldots, x_{n-2}, f(x_{n-1},x_n), g(x_{n-1},x_n)\big).
\end{equation*}
with
\begin{equation}\label{e:f}
f(x_{n-1},x_n) := \frac{\textup{sgn}(x_n)}{\sqrt{2}}\sqrt{\sqrt{x_{n-1}^2+x_n^2}-x_{n-1}}
\end{equation}
\begin{equation}\label{e:g}
g(x_{n-1},x_n) := \frac{1}{\sqrt{2}}\sqrt{\sqrt{x_{n-1}^2+x_n^2}+x_{n-1}}
\end{equation}
and
\[
\textup{sgn}(x_n) =
\begin{cases}
1 & \text{if } x_n\geq 0\\
1 & \text{if } x_n< 0.
\end{cases}
\]
Therefore, for every $x \in B_1 \setminus \Bnm$
we have
\begin{align}
z_\infty &(x) = u_0 \circ \Phi^{-1} (x) = g(x_{n-1},x_n)
\left(b_0 + \sum_{i=1}^{n-2}b_i\, x_i\right)\label{e:exp1}\\
&+ g(x_{n-1},x_n) \Big( b_{n-1}\, f(x_{n-1},x_n)
+b_n\, g(x_{n-1},x_n) + \sum_{i,j=1}^{n-2} b_{i,j} x_i\,x_j\Big)\label{e:exp2}\\
&+ 2\,g(x_{n-1},x_n) \sum_{i=1}^{n-2} \big(b_{i,n-1 }\, x_i\,  f(x_{n-1}, x_n)
+ b_{i,n }\, x_i\, g(x_{n-1}, x_n)\big)\label{e:exp3}\\
&+ b_{n-1,n-1}\,g(x_{n-1},x_n) \,f^2(x_{n-1},x_n) 
+ 2\, b_{n-1,n}\, f(x_{n-1},x_n)\, g^2(x_{n-1},x_n) \notag\\
&+ b_{n,n}\,g^3(x_{n-1},x_n) + H(x),\label{e:exp6}
\end{align}
with
\[
|H(x)| \leq C\, |x|^2 \quad \forall \;x \in B_1\setminus \Bnm
\]
for some $C>0$.

Due to the $\sfrac32$-homogeneity of $z_\infty$ and the $\sfrac12$-homogeneity
of $f$ and $g$, we deduce that the first term in \eqref{e:exp1}, as well as
the first two terms of \eqref{e:exp2}, \eqref{e:exp3} and 
the function $H$ in \eqref{e:exp6}
have the wrong homogeneity and therefore are identically zero, thus reducing the expansion
of $z_\infty$ to the following
\begin{align}
z_\infty (x) &= g(x_{n-1},x_n)\sum_{i=1}^{n-2}b_i\, x_i
+ b_{n-1,n-1}g(x_{n-1},x_n) \,f^2(x_{n-1},x_n)\notag\\
&+ 2\, b_{n-1,n}\, f(x_{n-1},x_n)\, g^2 (x_{n-1},x_n)
+ b_{n,n}g^3(x_{n-1},x_n). \label{e:exp7}
\end{align}
In addition, \eqref{e:f} and \eqref{e:g} yield
\[
f(x_{n-1},x_n) \, g(x_{n-1},x_n) = \frac{x_n}{2},
\]
and, in turn, plugging the latter identity in \eqref{e:exp7} implies
\begin{multline}\label{e:exp8}
z_\infty (x)=g(x_{n-1},x_n)\cdot\\\cdot\Big(\sum_{i=1}^{n-2}b_i\, x_i
+ b_{n-1,n-1} \,f^2(x_{n-1},x_n)
+ b_{n-1,n}\, x_n + b_{n,n}\,g^2(x_{n-1},x_n)\Big).
\end{multline}
To conclude the proof of Step~8 we need only to check for which choices
of the coefficients the right hand side of \eqref{e:exp8} is harmonic.
To this aim we notice that $g$ is itself a harmonic function, i.e.~$\Delta g = 0$ 
in $B_1\setminus \Bnm$. Therefore, we compute $\Delta z_\infty$ thanks to 
\eqref{e:f}, \eqref{e:g} and \eqref{e:exp8} 
as follows
\begin{gather}
 \Delta z_\infty(x)=
\frac{(3\,b_{n,n} + b_{n-1,n-1})
\sqrt{2\,\sqrt{x_{n-1}^2 + x_n^2}+x_{n-1}}}{\sqrt{x_{n-1}^2 + x_n^2}}
\notag\\
+ \frac{b_{n-1,n}\,x_n}{\sqrt{x_{n-1}^2 + x_n^2}\,\sqrt{
\sqrt{x_{n-1}^2 + x_n^2}+x_{n-1}}}.\label{e:conto laplaciano}
\end{gather}
Note that the function on the right hand side of \eqref{e:conto laplaciano} 
is identically zero on $B_1\setminus \Bnm$ if and only if
\[
3\,b_{n,n} + b_{n-1,n-1} =0\quad \text{and}\quad b_{n-1,n}=0.
\]
Thus, coming back to \eqref{e:exp8} we conclude that
\begin{align*}
z_\infty (x) &= g(x_{n-1},x_n)\,\Big(\sum_{i=1}^{n-2}b_i\, x_i
-3\,b_{n,n}\,f^2(x_{n-1},x_n) + b_{n,n}\,g^2(x_{n-1},x_n)\Big)\notag\\
& \stackrel{\eqref{e:f},\,\eqref{e:g}}{=} g(x_{n-1},x_n)\sum_{i=1}^{n-2}b_i\, x_i+
b_{n,n}\,g(x_{n-1},x_n) \left(2\,x_{n-1} - \sqrt{x_{n-1}^2 + x_n^2}
\right),
\end{align*}
which is the desired formula \eqref{e:rigidity} for $a_0 = b_{n,n}$ and 
$a_i = b_i$ for $i=1,\ldots, n-2$.

\medskip

\noindent{\bf Step 9: Case (2) of Step~3 cannot occur.} 
We finally reach a contradiction by excluding
also case (2) in Step~3.
 \smallskip
 
We use the orthogonality condition derived in Step~7, i.e.,
\begin{equation}\label{e:orto}
 \langle z_\infty, \zeta\rangle=0\qquad\textrm{ for all $\zeta\in T_{h}\mathscr{H}_{\sfrac32}$.}
\end{equation}
Since $z_\infty$ has the form in \eqref{e:rigidity},
we can choose $h$ as test function 
in \eqref{e:orto} to deduce $a_0=0$. Then take
$\zeta=v_{e_{n-1},\xi}$ (cp.~\eqref{e:vnuxi}) to deduce $a_1=\ldots=a_{n-2}=0$ by the 
arbitrariness of $\xi\in\mathbb{S}^{n-1}$ with $\xi\cdot e_n=\xi\cdot e_{n-1}=0$.

Therefore, $z_\infty$ is the null function, contradicting \eqref{e:zinfty}.
In this way we have excluded all the cases of Step~3 and conclude the proof of the
theorem.
\end{proof}
\medskip

\subsection{Even frequencies: the lowest stratum of the singular set}

In this subsection we prove Theorem~\ref{t:epithin2m}. 
As already remarked the arguments are similar to those of 
Theorem~\ref{t:epithin32}, some simplifications are actually occurring. Therefore, 
we shall only underline the substantial changes. Further, we keep 
the notation introduced in Theorem~\ref{t:epithin32} except for $\GGG$, that 
in the ensuing proof stands for $\GGG_{2m}$. 
\begin{proof}[Proof of Theorem~\ref{t:epithin2m}]
We start off as in  Theorem~\ref{t:epithin32} with a contradiction argument assuming the existence 
of $\kappa_j,\,\delta_j\downarrow 0$, and of $c_j\in H^1(B_1)$ such that $c_j$ is $2m$-homogeneous, 
$c_j\geq 0$ on $\Bn$, 
\begin{equation}\label{e:distc2m}
{\mathrm{dist}_{H^1}}\big(c_j,\mathscr{H}_{2m}\big)=\delta_j,
\end{equation}
\begin{equation}\label{e:epic2m}
 (1-\kappa_j)\GGG(c_j) \leq \inf_{v\in \mathscr{A}_{c_j}}\GGG(v)
\end{equation}
and $\psi_j:=P(c_j)\in \mathscr{H}_{2m}^{(0)}$ with
\begin{equation}\label{e:distc12m}
 \|c_j-\psi_j\|_{H^1}={\mathrm{dist}_{H^1}}\big(c_j,\mathscr{H}_{2m}\big)=\delta_j\qquad\text{for all $j \in \N$}.
\end{equation}

\medskip

We divide the rest of the proof in some intermediate steps corresponding to those of Theorem~\ref{t:epithin32}.

\medskip

\noindent{\bf Step 1: Introduction of a family of auxiliary functionals.}
We rewrite inequality \eqref{e:epic2m} conveniently and 
interpret it as an almost minimality condition.

\smallskip

For fixed $j$, we use \eqref{e:firstvar2m} and algebraic manipulations to rewrite \eqref{e:epic2m} 
for all $v \in \mathscr{A}_{c_j}$ as 
\begin{equation}\label{e:quasi minchia2m}
(1-\kappa_j)\GGG(c_j-\psi_j)\leq \GGG(v-\psi_j).
\end{equation}
Setting 
\begin{equation}\label{e:z_j2m}
z_j:=\frac{c_j-\psi_j}{\delta_j},
\end{equation}
\[
\hat{\psi}_j:=\frac{\psi_j}{\|\psi_j\|_{H^1}}\quad \text{ and }\quad
\vartheta_j:=\frac{\|\psi_j\|_{H^1}}{\delta_j},
\]
\eqref{e:quasi minchia2m} reduces to
\begin{equation}\label{e:epic22m}
(1-\kappa_j)\GGG_j(z_j)\leq \GGG_j(z)\quad \text{ for all }z\in {\mathscr{B}}_{j},
\end{equation}
where $\GGG_j : L^2(B_1) \to (-\infty, +\infty]$ is given by
\begin{equation}\label{e:Gj2m}
\GGG_j(z):=
\begin{cases}
\displaystyle{\int_{B_1}|\nabla z|^2dx-2m\int_{\partial B_1}z_j^2\,d\cHn} & 
\text{if $z\in {\mathscr{B}_{j}}$,} \\
+\infty  & \textup{otherwise,}
\end{cases}
\end{equation}
and
\begin{equation}\label{e:Bj2m}
{\mathscr{B}}_{j}:=\left\{z\in z_j+H^1_0(B_1):\,(z+\vartheta_j\hat{\psi}_j)\vert_{\Bn}\geq 0\right\}.
\end{equation}
Moreover, note that by \eqref{e:distc12m} and \eqref{e:z_j2m}
\begin{equation}\label{e:distc32m}
 \|z_j\|_{H^1(B_1)}=1.
\end{equation}
This implies that we can extract a subsequence (not relabeled) such that
\begin{itemize}
\item[(a)] $(z_j)_{j \in \N}$ converges weakly in $H^1(B_1)$ to some $z_\infty$;
\item[(b)] the corresponding traces $(z_j\vert_{\Bn})_{j \in \N}$;
 converge strongly in $L^2(\Bn)\cup L^2(\partial B_1)$;
\item[(c)] $(\vartheta_j)_{j\in\N}$ has a limit $\vartheta\in[0,\infty]$;
\item[(d)] $(\hat{\psi}_j)_{j \in \N}$ converges in $H^1(B_1)$ to some non-trivial $2m$-homogeneous 
harmonic polynomial $\psi_\infty$ satisfying $\psi_\infty\geq 0$ on $\Bn$. Actually, being the 
$\psi_j$'s polynomials the convergence occurs in any $C^\ell_{}(B_1)$ norm.
\end{itemize}

\medskip

\noindent{\bf Step 2: First properties of $(\GGG_j)_{j\in\N}$.}
We establish the equi-coercivity and some further properties
of the family of the auxiliary functionals $(\GGG_j)_{j\in\N}$.

\smallskip

In this case equi-coercivity is a straightforward consequence of definition \eqref{e:Gj2m} and 
item (b) in Step~1: for $j$ sufficiently large and for all $w \in H^1(B_1)$
\begin{equation}\label{e:equi-coerc2m}
\int_{B_1}|\nabla w|^2dx-3m\int_{\partial B_1}z^2_\infty\leq\GGG_j(w).
\end{equation}
Moreover notice that \eqref{e:distc32m} and item (b) imply
\begin{equation}\label{e:Gjzj2m}
\sup_j|\GGG_j(z_j)|<\infty.
\end{equation}

\medskip

\noindent{\bf Step 3: Asymptotic analysis of $(\GGG_j)_{j\in\N}$.}
Here we prove a $\Gamma$-convergence result for the family of energies $\GGG_j$. 

\smallskip

More precisely, we distinguish two cases.

\begin{itemize}
\item[(1)] If $\vartheta \in [0, +\infty)$, then
\[
(z_\infty+\vartheta\,\psi_\infty)\vert_{\Bn}\geq0,
\]
and $\Gamma(L^2(B_1))\hbox{-}\lim_j \GGG_j = \GGG_\infty^{(1)}$, 
where
\[
\GGG_\infty^{(1)}(z):=
\begin{cases}
\displaystyle{
\int_{B_1}|\nabla z|^2dx-2m\int_{\partial B_1}z_\infty^2\,d\cHn} & \text{if $z \in {\mathscr{B}^{(1)}_{\infty}}$,}\\
+\infty  &  \textup{otherwise,}
\end{cases}
\]
and ${\mathscr{B}_{\infty}^{(1)}} := \{z\in z_\infty+H^1_0(B_1):\, (z+\vartheta\,\psi_\infty)\vert_{\Bn}\geq0\}$.

\item[(2)] If $\vartheta = +\infty$ , $\Gamma(L^2(B_1))\hbox{-}\lim_j \GGG_j=\GGG_\infty^{(2)}$,
where
\[
\GGG_\infty^{(2)}(z):=
\begin{cases}
\displaystyle{
\int_{B_1}|\nabla z|^2dx-2m\int_{\partial B_1}z_\infty^2\,d\cHn} & 
\text{if $z \in {\mathscr{B}^{(2)}_{\infty}}$,}\\
+\infty  & \textup{otherwise,}
\end{cases}
\]
and ${\mathscr{B}^{(2)}_{\infty}} := z_\infty+H^1_0(B_1)$.
\end{itemize}

\medskip

\noindent{\emph{Proof of the $\Gamma$-convergence in case (1).}}
The fact that $z_\infty \in \mathscr{B}_\infty^{(1)}$ and the $\liminf$ inequality \eqref{e:liminf} 
can be deduced exactly as in case (1) of Theorem~\ref{t:epithin32}.

\smallskip

Instead, for what concerns the $\limsup$ inequality \eqref{e:limsup}, after performing the reductions 
in the corresponding step of Theorem~\ref{t:epithin32}, one has to take
\[
w_k^r:=\varphi\,(w+\vartheta\,\psi_\infty-\vartheta_{j_k}\,\hat{\psi}_{j_k})+(1-\varphi)\,z_{j_k}
\]
to infer \eqref{e:limsupr}. The conclusion then follows by a diagonalization argument.

\noindent{\emph{Proof of the $\Gamma$-convergence in case (2).}}
The fact that $z_\infty \in \mathscr{B}_\infty^{(2)}$ and the $\liminf$ inequality \eqref{e:liminf} are 
simple consequences of the lower semicontinuity of the Dirichlet energy under $L^2$ convergence and 
of the equi-coercivity of $(\GGG_j)_{j\in\N}$ in \eqref{e:equi-coerc2m}.
\smallskip

For the proof of the the $\limsup$ inequality \eqref{e:limsup} we need first to observe that the set
$\{\psi_\infty=0\}\cap\Bn$ is contained in a $(n-2)$-dimensional subspace $H$.
In order to prove this claim, we start noticing that
$\psi_\infty$ is a 
non-trivial $2m$-homogeneous harmonic polynomial satisfying $\psi_\infty\geq 0$ on $\Bn$. Therefore, 
its zero set is a cone containing the origin. Moreover, $\psi_j|_{\Bn}$ 
is convex by \cite[Theorem~1]{Athan-Caff}: this implies that $\{\psi_\infty=0\}\cap\Bn$ is 
actually a convex cone. Since $\psi_\infty|_{\Bn}\geq 0$ we infer from \eqref{e:dernorm2m} that
$\{\psi_\infty=0\}\cap\Bn\subseteq\{\psi_\infty=|\nabla\psi_\infty|=0\}$. Thus the harmonicity of 
$\psi_\infty$ finally yields the claim thanks to \cite[Theorem 3.1]{Caffarelli-Friedman}
(see also \cite[Lemma~1.9]{Hardt-Simon}).

Given this, the construction of the recovery sequence is analogous to that of case (2) 
in Theorem~\ref{t:epithin32}. With $K:=\partial B_1$, let $G_\eps$ and $\phi_\eps$ be defined
correspondingly, then $w^{\eps,\delta}:=\phi_\eps\,v+(1-\phi_\eps)w$ satisfies
\eqref{e:step1-regolarizzazione1}-\eqref{e:step1-regolarizzazione2}. Finally, assuming 
$H\subseteq\{x_{n-1}=x_n=0\}$, and setting 
$w^{\eps,\delta}_k:=\chi_\eps\,w^{\eps,\delta}+(1-\chi_\eps)z_{j_k}$ we conclude by choosing
$\delta=\eps^2$ and by a suitable diagonalization argument.

\medskip

\noindent{\bf Step 4: Improving the convergence of $(z_j)_{j\in \N}$.}
$(z_j)_{j\in\N}$ converges strongly to $z_\infty$ in $H^1(B_1)$.

\smallskip

This follows thanks to \eqref{e:Gjzj2m} and by standard $\Gamma$-convergence arguments
(cp. the corresponding step of Theorem~\ref{t:epithin32}). In particular, $\|z_\infty\|_{H^1}=1$.

\medskip

To reach the final contradiction we exclude next both the instances in Step~3 above.

\medskip

\noindent{\bf Step 5: Case (1) cannot occur.} 
We recall what we have achieved so far about $z_\infty$, namely
\begin{itemize}
\item[(i)] $\|z_\infty\|_{H^1}=1$,
\item[(ii)] $z_\infty$ is $2m$-homogeneous and even with respect to $x_n=0$,
\item[(iii)] $z_\infty$ is the unique minimizer of $\GGG_\infty^{(1)}$ with respect
to its own boundary conditions,
\item[(iv)] $z_\infty\in\mathscr{B}_\infty^{(1)}$, i.e.~$z_\infty +\vartheta\,\psi_\infty\geq 0$
on $\Bn$.
\end{itemize}

\smallskip

Arguing as in Step~5 of Theorem~\ref{t:epithin32}, the properties above and a change of variables for
$\GGG_\infty^{(1)}$ show that $w_\infty:=z_\infty+\vartheta \psi_\infty$ minimizes the Dirichlet energy 
among all maps $w \in w_\infty + H_0^1(B_1)$ such that $w\geq 0$ in $\Bn$ in the sense of traces. 

Therefore, $w_\infty$ is a $2m$-homogeneous solution of the Signorini problem, that is 
$w_\infty\in \mathscr{H}_{2m}$ by \cite[Lemma~1.3.3]{Garofalo-Petrosyan} (see also
\cite[Proposition~9.11]{PSU}). 
We get a contradiction as
\begin{equation}\label{e:celopiccolo2m}
\frac{c_j}{\delta_j} = \frac{\psi_j}{\delta_j}+ z_j=\vartheta_j \hat{\psi}_j + z_j
\to  \vartheta\, \psi_\infty + z_\infty = w_\infty \in\mathscr{H}_{2m} 
\quad \text{ in $H^1(B_1)$,}
\end{equation}
which implies, for $j$ sufficiently large,
\begin{gather*}
\dist_{H^1}(c_j, \mathscr{H}_{2m})\leq \|c_j-\delta_jw_\infty\|_{H^1(B_1)} 
\stackrel{\eqref{e:celopiccolo2m}}{=} o(\delta_j)
< \delta_j = \dist_{H^1}(c_j, \mathscr{H}_{2m}),
\end{gather*}
having used in the last line that $\delta_j w_\infty \in \mathscr{H}_{2m}$.

\medskip

We are then left with excluding case (2) of Step~3 to end the contradiction argument.
Since in the present proof the Step~3 has two cases rather than three, for a more clear 
comparison with Theorem~\ref{t:epithin32} we number the next step as 7 rather than 6.
\medskip

\noindent{\bf Step 7: An orthogonality condition.}
We exploit the fact that $\psi_j$ is the point of minimal distance of $c_j$ 
from $\mathscr{H}_{2m}$ to deduce that $z_\infty$ is orthogonal to $\widehat{\mathscr{H}}_{2m}$.

\smallskip

The very definitions of $\psi_j$ in \eqref{e:distc12m} and of $z_j$ in \eqref{e:z_j2m} 
imply that for all $\psi\in\mathscr{H}_{2m}$
\begin{equation*}
-\|\psi_j-\psi\|^2_{H^1(B_1)}\leq 2\delta_j
\langle z_j, \psi_j-\psi\rangle.
\end{equation*}
Therefore, assuming $\psi_j\neq \psi$ and renormalizing, 
we get 
\[
-\|\psi_j-\psi\|_{H^1}\leq 2\delta_j\langle z_j, \frac{\psi_j-\psi}{\|\psi_j-\psi\|_{H^1}}\rangle,
\]
and by taking the limit $\psi\to \psi_j$ in $H^1$ we conclude that 
\[
\langle z_j, \zeta\rangle=0
\qquad\textrm{ for all $\zeta\in T_{\psi_j}\mathscr{H}_{2m}
\stackrel{\eqref{e:tg2m0}}{=}\widehat{\mathscr{H}}_{2m}$.}
\]
Now letting $j\uparrow\infty$ in the equality above we get that
\begin{equation}\label{e:perp12m}
\langle z_\infty, \zeta\rangle=0\qquad\textrm{ for all $\zeta\in\widehat{\mathscr{H}}_{2m}$}.
\end{equation}

\medskip

\noindent{\bf Step 8: Identification of $z_\infty$ in case (2) of Step~3.}
We have that 
\begin{equation}\label{e:rigidity2m}
z_\infty\in \widehat{\mathscr{H}}_{2m}\setminus\{0\}.
\end{equation}

\smallskip

Indeed we have already shown that
\begin{itemize}
\item[(i)] $\|z_\infty\|_{H^1}=1$,
\item[(ii)] $z_\infty$ is $2m$-homogeneous and even with respect to $x_n=0$,
\item[(iii)] $z_\infty$ is the unique minimizer of $\GGG_\infty^{(2)}$
with respect to its own boundary conditions, i.e.~is harmonic.
\end{itemize}

\medskip

\noindent{\bf Step 9: Case (2) of Step~3 cannot occur.} 
We finally reach a contradiction by excluding also case (2) in Step~3.

\smallskip

Because of \eqref{e:rigidity2m} we can choose $z_\infty$ itself as a test function in 
\eqref{e:perp12m} to deduce that it is actually the null function, thus contradicting 
\eqref{e:zinfty}.

In this way we have excluded all the cases of Step~3 and we can conclude the proof of the
theorem.
\end{proof}

%
%
\section{Regularity of the free boundary}\label{s:free boundary}

In this section we show how to derive the regularity of the free boundary around
points of least frequency as a simple consequence of the epiperimetric inequality.
To this aim we need to recall some notation and some results from the literature.
Since we are going to use the monotonicity formulas proven in
\cite{Garofalo-Petrosyan}, we try to follow the notation therein as closely
as possible.
In what follows $u\in H^1(B_1)$ shall denote a solution to the Signorini
problem (even symmetric respect to $\{x_n=0\}$).
We denote with $\Lambda(u)$ the coincidence set, i.e.
\[
\Lambda(u) := \big\{(\hat x, 0) \in \Bn\,: u(\hat x, 0) =0 \big\}
\]
and by $\Gamma(u)$ the free boundary of $u$, namely the topological boundary
of $\Lambda(u)$ in the relative topology of $\Bn$.

For $x_0 \in \Gamma(u)$ and $0<r < 1 - |x_0|$
let $N^{x_0}(r,u)$ be the \textit{frequency function} defined by
\[
N^{x_0}(r,u) := \frac{r\int_{B_r(x_0)}|\nabla u|^2\,dx}{\int_{\de B_r(x_0)} u^2\,d\cHn}
\]
provided $u\vert_{\de B_r(x_0)} \not\equiv 0$.
As proven in \cite{Athan-Caff-Sal} the function $(0, \dist(x_0,\de B_1))
\ni r\mapsto N^{x_0}(r, u)$ is nondecreasing for every $x_0 \in \Bn$.
It is then possible to define
the limit $N^{x_0}(0^+, u) := \lim_{r \downarrow 0} N^{x_0}(r, u)$ and
as shown in \cite{Athan-Caff-Sal} it holds $N^{x_0}(0^+, u) \geq \sfrac32$
for every $x_0 \in \Gamma(u)$.
We then denote with $\Gamma_{\sfrac32}$ the points of the free boundary with minimal frequency, also called 
\textit{regular points} in \cite{Garofalo-Petrosyan}:
\[
\Gamma_{\sfrac32} := \big\{x_0 \in \Gamma(u) \,: N^{x_0}(0^+,u) = \sfrac32 \big\}.
\]
Note that by the monotonicity of the frequency it follows that
$\Gamma_{\sfrac32}(u) \subset \Gamma(u)$ is open in the relative
topology.
We also introduce the shorthand notation
\[
D^{x_0}(r) := \int_{B_r(x_0)} |\nabla u|^2\,dx \quad
\text{and}\quad 
H^{x_0}(r) := \int_{\de B_r(x_0)} u^2\,d\cHn
\]
and we always omit to write the point $x_0$ in the notation if $x_0 =0$.
Finally we recall the following simple consequence of the monotonicity of the frequency
function, proven in \cite[Lemma~1]{Athan-Caff-Sal}:
the function $(0,1-|x_0|)\ni r\mapsto \frac{H^{x_0}(r)}{r^{n+2}}$ is nondecreasing
and in particular
\begin{equation}\label{e:H dall'alto}
H^{x_0}(r) \leq \frac{H(1-|x_0|)}{(1-|x_0|)^{n+2}}\, r^{n+2} \quad\forall \; 0<r<1-|x_0|,
\end{equation}
and for every $\eps>0$ there exists $r_0(\eps)>0$ such that
\begin{equation}\label{e:H dal basso}
H(r) \geq \frac{H(r_0)}{r_0^{n+2+\eps}}\, r^{n+2+\eps} \quad\forall \; 0<r<r_0.
\end{equation}
For readers' convenience we provide a short account of these statements
in Appendix~\ref{a:frequency}.

\subsection{Decay of the boundary adjusted energy}
The boundary adjusted energy $\GGG$ considered above is the unscaled 
version of a member of a family of energies \textit{\`a la} Weiss  
introduced in \cite{Garofalo-Petrosyan}, namely the one corresponding
to the lowest frequency.
Since we also need to consider its rescaled version, we shift to the notation
used in \cite{Garofalo-Petrosyan}: $\GGG(u) = W_{\sfrac32}(1, u)$ with
\[
W_{\sfrac32}^{x_0} (r, u) := \frac{1}{r^{n+1}} \int_{B_r(x_0)}|\nabla u|^2\,dx - 
\frac{3}{2\,r^{n+2}}\int_{\de B_r(x_0)} u^2\, d\cHn.
\]
Note that by the monotonicity result in \cite{Garofalo-Petrosyan} it follows that
\begin{equation}\label{e:W'_0}
\frac{d}{dr} W_{\sfrac32} (r, u)  =
\frac{2}{r} \int_{\de B_1} \left(\nabla u_r \cdot \nu - \frac32 u_r\right)^2 d\cHn. 
\end{equation}
and for every $x_0 \in \Gamma_{\sfrac32}(u)$ it holds that
$W_{\sfrac32}^{x_0} (r, u) \geq 0$ for every $0<r< 1 - |x_0|$ with
\[
\lim_{r\to 0} W_{\sfrac32}^{x_0} (r, u) = 0.
\]

In this subsection we show the main consequence of the epiperimetric inequality
in Theorem~\ref{t:epithin32}, namely the decay of the boundary adjusted energy.
The next lemma enable us to apply the epiperimetric inequality uniformly
at any free boundary point in $\Gamma_{\sfrac32}(u)$.

\begin{lemma}\label{l:cpt}
Let $u\in H^1(B_1)$ be a solution to the Signorini problem.
Then for every compact set $K \subset \subset \Bn$ and
for every $\eta>0$ there exists $r_0 >0$ such that
\begin{equation}\label{e:smallness}
\dist(c^{x_0}_r, \mathscr{H}_{\sfrac32}) \leq \eta
\quad \forall\; x_0 \in \Gamma_{\sfrac32}(u) \cap K,
\quad\forall \; 0<r<r_0,
\end{equation}
where $c^{x_0}_r \in H^1(B_1)$ is given by
\[
c^{x_0}_r(x) := \frac{|x|^{\sfrac32}}{r^{\sfrac32}} u\Big(\frac{rx}{|x|}+x_0\Big).
\]
\end{lemma}

\begin{proof}
The proof is done by contradiction: we assume that there exist $\eta>0$
and sequences of points $x_k \in K \cap \Gamma_{\sfrac32}(u)$ and radii
$r_k \downarrow 0$ such that
\begin{equation}\label{e:smallness contradicted}
\dist(c^{x_k}_{r_k}, \mathscr{H}_{\sfrac32}) > \eta \quad \forall\; k\in \N.
\end{equation}
Let us introduce the following rescaled function (cp.~next subsection for further discussions)
\begin{equation*}
u^{x_0}_{r}(x) := \frac{u(x_0 + r x)}{r^{\sfrac32}} 
\quad \forall\; 0<r<1-|x_0|, 
\quad \forall \;x\in B_{\sfrac{(1-|x_0|)}{r}}.
\end{equation*}
It follows from \eqref{e:H dall'alto} that
\[
\sup_{k \in \N} \|u^{x_k}_{r_k}\|_{L^2(B_1)} < +\infty,
\]
from which (by the regularity for the solution of the Signorini problem --
cp.~Appendix~\ref{a:frequency}) we deduce
\begin{equation}\label{e:uniform reg}
\sup_{k \in \N} \|u^{x_k}_{r_k}\|_{C^{1,\sfrac12}(B_1)} < +\infty.
\end{equation}
We can then conclude that up to passing to a subsequence (not relabeled)
there exists a function $w_0 \in C^{1,\sfrac12}(B_1)$ such that
$\|u^{x_k}_{r_k}-w_0\|_{C^{1,\alpha}(B_1)} \to 0$ for every $\alpha <\sfrac12$.
Moreover without loss of generality we can also assume that $x_k \to x_0 \in K \cap \Gamma_{\sfrac32}$.

By a simple argument (whose details are left to the reader) the function
$w_0$ is itself a solution to the Signorini problem. Moreover we claim that $w_0$
is $\sfrac32$-homogeneous. In order to show the claim we start noticing the following:
for every $\delta>0$ we can fix $\rho>0$ such that $N^{x_0}(\rho, u) \leq \sfrac32 + \delta$.
Therefore, for $k$ sufficiently large we infer that for every $s \in (0,1)$:
\begin{align}
N(s,u^{x_k}_{r_k}) & = N^{x_k}(s\,r_k,u)\notag\\
&= N^{x_k}(s\,r_k,u) - N^{x_k}(\rho,u)  + N^{x_k}(\rho,u)-N^{x_0}(\rho,u)+N^{x_0}(\rho,u)\notag\\
& \leq N^{x_k}(\rho,u)-N^{x_0}(\rho,u) + \frac32 +\delta\leq \frac32 +2\,\delta,
\end{align}
where we used the monotonicity of the frequency and the fact that for $k$ large enough
$|N^{x_k}(\rho,u)-N^{x_0}(\rho,u)| \to 0$ as $x_k \to x_0$.
In particular from the convergence of $u^{x_k}_{r_k}$ to $w_0$ and the arbitrariness
of $\delta$ we deduce that $N(s, w_0) \equiv \sfrac32$, i.e.~$w_0$ is $\sfrac32$-homogeneous.

From the already cited classification result in \cite[Theorem~3]{Athan-Caff-Sal} we infer that
$w_0 \in \mathscr{H}_{\sfrac32}$, thus clearly contradicting \eqref{e:smallness contradicted}.
\end{proof}

We are now in the position to prove the decay
of the boundary adjusted energy. To this aim
we recall some elementary formulas, whose verification is left to the 
reader
(details can be also found in \cite{Garofalo-Petrosyan}):
\begin{gather}
H'(r) := \frac{d}{dr} H(r) = \frac{n-1}{r} H(r) + 2\int_{\de B_r} u \nabla u\cdot \nu\,d\cHn,\label{e:H'}\\
D'(r) := \frac{d}{dr} D(r) = \frac{n-2}{r} D(r) + 2\int_{\de B_r} (\nabla u\cdot \nu)^2\,d\cHn,\label{e:D'}\\
D(r) = \int_{\de B_r} u\nabla u\cdot \nu\,d\cHn.\label{e:D}
\end{gather}
We mention once for all that all the integral quantities $D(r)$, $H(r)$ etc{\ldots}
considered in this section are absolutely continuous functions of the radius
and therefore can be differentiated almost everywhere.

\begin{proposition}\label{p:decay energy}
There exists a dimensional constant $\gamma>0$ with this property.
For every compact set $K \subset \Bn$ there exists a constant $C>0$
such that
\begin{equation}\label{e:decay energy}
W^{x_0}_{\sfrac32} (r, u) \leq C\, r^\gamma\quad\forall\;0<r<\dist(K, \de B_1), \; \forall \; x_0 \in 
\Gamma_{\sfrac32}\cap K.
\end{equation}
\end{proposition}

\begin{proof}
We start considering the case of $x_0 = 0 \in \Gamma_{\sfrac32}$.
We can then compute as follows:
\begin{align}
\frac{d}{dr} W_{\sfrac32} (r, u) & = -\frac{n+1}{r^{n+2}} D(r) + \frac{D'(r)}{r^{n+2}}
-\frac{3}{2r^{n+2}}H'(r)+\frac{3(n+2)}{2r^{n+3}}H(r)\notag\\
&= -\frac{n+1}{r^{}}  W_{\sfrac32} (r, u) - \frac{3(n+1)}{2r^{n+3}}H(r)\notag\\
& \quad\quad + \underbrace{\frac{D'(r)}{r^{n+2}} + \frac{9}{2r^{n+3}} H(r) - 3 \frac{D(r)}{r^{n+2}}}_{=:I}.\label{e:W'_1}
\end{align}
In order to treat the terms in $I$, we introduce the rescaled functions
\begin{equation}\label{e:u_r}
u_r(x) := \frac{u(rx)}{r^{\sfrac32}}
\end{equation}
and deduce by simple computations that
\begin{align}
I & = \frac{1}{r} \int_{\de B_1} \left(|\nabla u_r|^2 - 3 u_r \nabla u_r \cdot \nu + \frac{9}{2}u_r^2 \right) d\cHn \notag\\
& =\frac{1}{r} \int_{\de B_1} \left[\left(\nabla u_r \cdot \nu - \frac32 u_r\right)^2 
+ |\nabla_\theta u_r|^2 + \frac{9}{4} u_r^2 \right] d\cHn,\label{e:I}
\end{align}
where we denoted by $\nabla_\theta u_r$ the differential of $u_r$ in
the directions tangent to $\de B_1$.
Let $c_r$ be the $\sfrac32$-homogeneous extension of $u_r\vert_{\de B_1}$, i.e.
\[
c_r (x):= |x|^{\sfrac32} u_r\Big(\frac{x}{|x|}\Big).
\] 
It is simple to verify that
\[
\int_{\de B_1} \left(|\nabla_\theta u_r|^2 + \frac{9}{4} u_r^2 \right) = (n+1)\int_{B_1}|\nabla c_r|^2\,dx.
\]
We then conclude that
\begin{align}
\frac{d}{dr} W_{\sfrac32} (r, u) & = \frac{n+1}{r^{}} \Big( W_{\sfrac32} (1, c_r)  - W_{\sfrac32} (1, u_r) \Big) \notag\\
&\qquad+ 
\frac{1}{r} \int_{\de B_1} \left(\nabla u_r \cdot \nu - \frac32 u_r\right)^2 d\cHn\label{e:W'_2}.
\end{align}
By \eqref{e:smallness} there exists $r_0>0$ such that 
the epiperimetric inequality in Theorem~\ref{t:epithin32} can
be applied (recall that $\GGG(w) = W_{\sfrac32}(1,w)$
for every $w \in H^1(B_1)$). In view of \eqref{e:W'_0}, we then deduce that
\begin{align}
\frac{d}{dr} W_{\sfrac32} (r, u) & \geq 2\frac{n+1}{r^{}}\frac{\kappa}{1-\kappa} W_{\sfrac32} (1, u_r) \quad\forall\; 0<r<r_0.\label{e:W'_3}
\end{align}
Integrating this inequality we get
\begin{equation}\label{e:decay W}
W_{\sfrac32} (r, u) \leq W_{\sfrac32} (1, u)\, r^\gamma\quad\forall\;0<r<r_0,
\end{equation}
with $\gamma := 2(n+1) \sfrac{\kappa}{1-\kappa}$.

In order to conclude the proof it is enough to observe that for every 
$x_0 \in K\subset\subset \Bn$
the decay \eqref{e:decay W} can be derived by the same arguments
with a constant $C>0$ which depends only on $D(1)$ and $\dist(K, \de B_1)$.
\end{proof}

\begin{remark}
Note that \eqref{e:W'_2} also shows the monotonicity of the boundary adjusted
energy: using the minimizing property for the Dirichlet energy
of $u_r$ with respect to functions with the same trace, we deduce indeed the less refined
with respect to \eqref{e:W'_0} estimate
\begin{align*}
\frac{d}{dr} W_{\sfrac32} (r, u)  
&\geq \frac{1}{r} \int_{\de B_1} \left(\nabla u_r \cdot \nu - \frac32 u_r\right)^2 d\cHn >0.
\end{align*}
\end{remark}

\subsection{Rescaled profiles}
In order to be able to use the computations above in the study of the property
of the free boundary, we need to consider the following limiting profiles of our
solution $u$: under the assumption that $0 \in \Gamma_{\sfrac32}(u)$ we set
\begin{equation}\label{e:u_r bis}
u_r(x) := \frac{u(rx)}{r^{\sfrac32}}
\end{equation}
Note that the rescalings in \eqref{e:u_r bis} are different from those
considered in the literature, e.g. in \cite{Athan-Caff-Sal}.

In order to deduce the existence of nontrivial blowups under the rescaling in
\eqref{e:u_r bis} we need to prove some growth estimates on the solution
to the Signorini problem.

A first consequence of \eqref{e:H dall'alto} is that the rescaled profiles $u_r$ have
equi-bounded Dirichlet energies:
\begin{align}
\int_{B_1} |\nabla u_r|^2 \, dx & = \frac{\int_{B_r}|\nabla u|^2\,dx}{r^{n+1}} =
\frac{r\int_{B_r}|\nabla u|^2\,dx}{\int_{\de B_r} u^2\,d\cHn} \, \frac{\int_{\de B_r} u^2\,d\cHn}{r^{n+1}}\notag\\
&\stackrel{\eqref{e:H dall'alto}}{\leq} N(r,u) \, H(1) \leq N(1,u) \, H(1).\label{e:equi energy} 
\end{align}
Therefore, for every infinitesimal sequence of radii $r_k \downarrow0$
there exists a subsequence $r_{k'}\downarrow 0$ such that
$u_{r_{k'}} \to u_0$ in $L^2(B_1)$.

Note however that \eqref{e:H dal basso} is not enough to deduce that there exists a limiting
function $u_0$ which is not identically $0$.
As an application of the epiperimetric inequality and the related decay of
the energy in Proposition~\ref{p:decay energy} we can deduce that this is actually the case
for every such limiting profiles $u_0$.

\begin{proposition}[(Nondegeneracy)]\label{p:nondegen}
Let $u \in H^1(B_1)$ be a solution to the Signorini problem
and assume that $0 \in \Gamma_{\sfrac32}(u)$. 
Then there exists a constant $H_0>0$ such that
\begin{equation}\label{e:nondegeneracy_1}
H(r) \geq H_0\, r^{n+2} \quad\forall \; 0<r<1.
\end{equation}
\end{proposition}

\begin{proof}
The starting point is the computation of the derivative in \eqref{e:log H}:
\begin{align}\label{e:log H bis}
\frac{d}{dr} \left(\log \frac{H(r)}{r^{n+2}} \right) 
= 2\frac{D(r)}{H(r)} - \frac{3}{r} = \frac{2\,r^{n+1}}{H(r)} W_{\sfrac32}(r,u).
\end{align}
Let $\gamma>0$ be the constant in Proposition~\ref{p:decay energy}
and let $\eps = \sfrac{\gamma}{2}$ in Lemma~\ref{l:altezza}.
Then by using \eqref{e:decay energy} and \eqref{e:H dal basso} in conjunction
we infer from \eqref{e:log H bis} that there exists a constant $C=C(\gamma)>0$ such that
\begin{align}
\frac{d}{dr} \left(\log \frac{H(r)}{r^{n+2}} \right) \leq C\,r^{\sfrac{\gamma}{2} - 1} \quad\forall\; 0<r<r_0,
\end{align}
where $r_0 = r_0(\eps)>0$ is the constant given by Lemma~\ref{l:altezza}.
Integrating this differential inequality we get that the function
\[
\frac{H(r)}{r^{n+2} e^{\frac{2C}{\gamma} r^{\sfrac{\gamma}{2}}}}
\]
is nonincreasing.
In particular we conclude that
\begin{gather}
\lim_{r\to0} \frac{H(r)}{r^{n+2} e^{Cr^{\sfrac{\gamma}{2}}}} = \lim_{r\to0}\frac{H(r)}{r^{n+2}} =: H_0
\end{gather}
and for sufficiently small $r>0$
\begin{gather}
H_0 \geq
\frac{H(r)}{r^{n+2}e^{\frac{2C}{\gamma} r^{\sfrac{\gamma}{2}}}} \geq
\frac{H(r)}{r^{n+2}} (1-\frac{2C}{\gamma} r^{\sfrac{\gamma}{2}}) >0
\end{gather}
where we used the elementary inequality $e^x \leq \frac{1}{1-x}$ for $x>0$ sufficiently small.
By the monotonicity of the function $\frac{H(r)}{r^{n+2}}$ proven in 
Lemma~\ref{l:altezza}
we then conclude.
\end{proof}

Note now that by \eqref{e:nondegeneracy_1} we then deduce that
\[
\int_{\de B_1} u_r^2\,d\cHn \geq H_0.
\]
Therefore, since from \eqref{e:equi energy} we also deduce the convergence of the traces of $u_r$ on $\de B_1$,
we finally get that
\[
\int_{\de B_1} u_0^2\,d\cHn \geq H_0>0
\]
for every limiting profile $u_0$, thus showing that $u_0 \not\equiv 0$.

\subsection{Uniqueness of blowups}
By compactness for every $x_0 \in \Gamma_{\sfrac32}(u)$ and
for every infinitesimal sequence $r_k\downarrow 0$ there exists
at least a subsequence (in the sequel not relabeled) such that
$u^{x_0}_{r_k} \to u^{x_0}_0$ in the weak topology of $H^1(B_1)$
for some $u^{x_0}_0\in H^1(B_1)$ nontrivial function.
These limiting functions $u^{x_0}_0$ are called in the sequel \textit{blowups}
of $u$ at the point $x_0$

It is very simple to show that $u^{x_0}_0$ are solutions to the Signorini problem.
Moreover $u^{x_0}_0$ is $\sfrac32$-homogeneous, 
i.e.~$x\cdot\nabla u^{x_0}_0 - \sfrac{3 u^{x_0}_0}{2} \equiv 0$.
Therefore, from the classification result in \cite[Theorem~3]{Athan-Caff-Sal}
we infer that $u^{x_0}_0 \in \mathscr{H}_{\sfrac32}$.

A key ingredient of the analysis of the free boundary we are going to perform
is to show that
\begin{itemize}
\item[(i)] the blowup $u_0$ to a solution of the Signorini problem
is actually \textit{unique}, meaning that the \textit{whole} sequence $u_r \to u_0$
in $L^2(B_1)$ as $r \to 0$,
\item[(ii)] there is a rate of convergence of the rescaled profiles to unique limiting blowup.
\end{itemize}

This is now an easy consequence of the epiperimetric inequality and it is shown in the next proposition.

\begin{proposition}\label{p:uniqueness blowup reg}
Let $u$ be a solution to the Signorini problem and $K \subset\subset \Bn$.
Then there exist a constant $C>0$ such
that for every $x_0\in \Gamma_{\sfrac32}(u)\cap K$
\begin{equation}\label{e:unique trace}
\int_{\de B_1} \left |u^{x_0}_{r} -u^{x_0}_0\right| \,d\cHn
\leq C\, r^{\sfrac{\gamma}{2}} \quad \textrm{for all }\; 0<r<\dist(K,\de B_1),
\end{equation}
where $\gamma>0$ is the constant in Proposition~\ref{p:decay energy}.
In particular, the blow-up limit $u_0$ at $x_0$ is unique.
\end{proposition}

\begin{proof}
Arguing as in the proof of Proposition~\ref{p:decay energy}
it is enough to show \eqref{e:unique trace} for $0 \in \Gamma_{\sfrac32}(u)$ and for a constant $C>0$ which depend only on the $L^2$ norm of $u$ and
on its Dirichlet energy.

Let $0<s<r <r_0$ be fixed radii with $r_0$ the constant in \eqref{e:smallness}
such that the epiperimetric inequality can be applied.
We can then use the formula \eqref{e:W'_2} to compute as follows:
\begin{align}\label{e:diff tracce}
\int_{\de B_1} |u_r - u_s|\, d\cHn & \leq
\int_{\de B_1} \int_{s}^r t^{-1}
\Big|\nabla u_t \cdot \nu - \frac32 u_t \Big|\,dt\, d\cHn\notag\\
& \leq \sqrt{n\,\omega_n}\int_{s}^r t^{-\sfrac12}
\left(t^{-1}\int_{\de B_1}\Big|\nabla u_t \cdot \nu - \frac32 u_t \Big|^2\, d\cHn\right)^{\sfrac12}dt\notag\\
& \stackrel{\eqref{e:W'_2}}{\leq} \sqrt{n\,\omega_n}
\int_s^r t^{-\sfrac12} \Big(\frac{d}{dt} W_{\sfrac32}(t, u)\Big)^{\sfrac12}\,dt\notag\\
&\leq \sqrt{n\,\omega_n}\, \log\big(\frac{r}{s}\big)\, \big(W_{\sfrac32}(r, u) - W_{\sfrac32}(s, u)\big)^{\sfrac12}.
\end{align}
By \eqref{e:decay energy} and a simple dyadic argument
(applying \eqref{e:diff tracce} to
$s=\sfrac{r}{2} = 2^{-k}$ for $k \in \N$ sufficiently large)
we easily deduce that for every $0<s<r<r_0$
\[
\int_{\de B_1} |u_r - u_s|\, d\cHn \leq
C\,r^{\sfrac{\gamma}{2}}
\]
for a constant $C>0$ which in turn depends only on the constants in
Proposition~\ref{p:decay energy}. Sending $s$ to $0$ and
eventually changing the value of the constant $C$,
we then conclude the proof of \eqref{e:unique trace}.
\end{proof}

\subsection{$C^{1,\alpha}$ regularity of the free boundary $\Gamma_{\sfrac32}$}
In view of the uniqueness result in Proposition~\ref{p:uniqueness blowup reg} we
are in the position to give a new proof of the $C^{1,\alpha}$ regularity
of the part of the free boundary with least frequency.

\begin{proposition}
Let $u \in H^1(B_1)$ be a solution to the Signorini problem.
Then there exists a dimensional constant $\alpha>0$
such $\Gamma_{\sfrac32}(u)$ is locally in $\Bn$ a $C^{1,\alpha}$ regular
submanifold of dimension $n-2$.
\end{proposition}

\begin{proof}
Without loss of generality it is enough to prove that if $0 \in \Gamma_{\sfrac32}(u)$
then $\Gamma_{\sfrac32}(u)$ is a $C^{1,\alpha}$ submanifold in 
a neighborhood of $0$.
To this aim we start noticing that by the openness of $\Gamma_{\sfrac32}(u)$
there exists $s>0$ such that $B_s\cap \Gamma(u) = B_s\cap \Gamma_{\sfrac32}(u)$.
Since for every $x_0 \in B_s\cap \Gamma_{\sfrac32}(u)$ the unique blowup of the
rescaled functions is of the form
\[
u^{x_0} = \lambda_{x_0} h_{e(x_0)} \in \mathscr{H}^{\sfrac32}
\]
for some $\lambda_{x_0}>0$ and $e(x_0) \in \bS^{n-2}$.

\medskip

We first prove the H\"older continuity of $x_0 \mapsto \lambda_{x_0}$.
To this aim we start observing that, thanks to Proposition~\ref{p:decay energy}
and Proposition~\ref{p:nondegen} we can further estimate \eqref{e:log H bis}
in the following way
\begin{equation}\label{e:log H tris}
\frac{d}{dr} \left(\log \frac{H^{x_0}(r)}{r^{n+2}} \right) 
= \frac{2\,r^{n+1}}{H^{x_0}(r)} W^{x_0}_{\sfrac32}(r,u) \leq C\, r^{\gamma-1}
\quad\forall r \in (0,1).
\end{equation}
Notice that by the strong convergence in $L^2(B_1)$ of the rescaled functions
it follows that
\[
\lambda_{x_0} = c_0\,\lim_{r\to 0} \frac{H^{x_0}(r)}{r^{n+2}}
\]
for some \textit{dimensional} constant $c_0>0$.
Integrating \eqref{e:log H tris} we can then deduce that
\begin{equation}\label{e:decay H}
c_0\,\frac{H^{x_0}(r)}{r^{n+2}} - \lambda_{x_0} \leq C\, r^\gamma \quad \forall\;
r\in (0, 1).
\end{equation}
Notice moreover that for $x_0, y_0 \in B_s \cap \Gamma_{\sfrac32}(u)$ and $r=|x_0-y_0|^{1- \theta}$ with $\theta = \sfrac{\gamma}{(1+\gamma)}$ it holds that
\begin{align}\label{e:rescaling vicini}
\int_{\de B_1}&|u^{x_0}_r - u^{y_0}_r|\,d\cHn
\notag\\
&\leq r^{-\sfrac32} \int_{\de B_1} \int_0^1 
\big|\nabla u\big(s(x_0 + rx) + (1-s)(y_0+rx) \big)|\,|y_0 - x_0|\,ds\,d\cHn(x)\notag\\
& \leq C r^{-1}\,|y_0 - x_0| \leq C\,|y_0 - x_0|^\theta.
\end{align}
Therefore we can conclude that for $r= |x_0-y_0|^{1-\theta}$
with $\theta = \sfrac{\gamma}{(1+\gamma)}$ it holds that
\begin{align}\label{e:lambda vicini}
\big|\lambda_{x_0} - \lambda_{y_0} \big| & \leq 
\Big\vert \lambda_{x_0} - c_0\frac{H^{x_0}(r)}{r^{n+2}}\Big\vert
+c_0\Big\vert \frac{H^{y_0}(r)}{r^{n+2}} - \frac{H^{y_0}(r)}{r^{n+2}} \Big\vert
+\Big\vert c_0\frac{H^{y_0}(r)}{r^{n+2}} - \lambda_{y_0}\Big\vert\notag\\
&\leq C\, r^{\gamma} + C\,\int_{\de B_1}\big|(u^{x_0}_r)^2 - (u^{y_0}_r)^2\big|\,d\cHn\notag\\
& \leq C\, r^{\gamma} + C \,\int_{\de B_1}\big|u^{x_0}_r - u^{y_0}_r\big|\,d\cHn
\leq C\, r^{\theta}
\end{align}
where we used the uniform $L^\infty$ (actually $C^{1,\sfrac12}$) bound
on $u^{x_0}_r$ for every $x_0 \in \Gamma_{\sfrac32}(u) \cap B_s$ (cp.~Appendix~\ref{a:frequency} for more details).

By Proposition~\ref{p:uniqueness blowup reg} and a similar computation
we can show that
\begin{align}\label{e:blowup vicini}
\int_{\de B_1} \left |u^{x_0}_{0} - u^{y_0}_0\right| \,d\cHn & \leq
\int_{\de B_1} \left |u^{x_0}_{0} - u^{x_0}_r\right| \,d\cHn
+\int_{\de B_1} \left |u^{x_0}_{r} - u^{y_0}_r\right| \,d\cHn\notag\\
&\qquad+\int_{\de B_1} \left |u^{y_0}_{r} - u^{y_0}_0\right| \,d\cHn\notag\\
&\stackrel{\eqref{e:unique trace}\,\& \,\eqref{e:rescaling vicini}}{\leq}
C\, r^{\sfrac{\gamma}{2}} + C\, |x_0 - y_0|^{\theta} \leq C\, |x_0 - y_0|^{\sfrac{\gamma\theta}{2}}
\end{align}

Note finally that there exists a geometric constant $\bar C>0$ such that
\[
|e(x_0) - e(y_0)| \leq \bar C\,\int_{\de B_1} \left |h_{e(x_0)} - h_{e(y_0)}\right| \,d\cHn.
\]
Therefore from \eqref{e:lambda vicini} and \eqref{e:blowup vicini}
we easily deduce that
\begin{equation}\label{e:normali vicine}
|e(x_0) - e(y_0)| \leq C\, |x_0 - y_0|^{\sfrac{\gamma\theta}{2}}.
\end{equation}

\medskip

Next we show that the vectors $e(x_0)$ do actually encode to a geometric
property of the free boundary. To this aim we introduce the following notation
for cones centered at points $x_0 \in \Gamma_{\sfrac32}(u)$: for
any $\eps >0$ we set
\[
C^\pm(x_0, \eps) := \left\{ x\in \R^{n-1} \times\{0\} \,:\, \pm\langle x-x_0, e(x_0)\rangle \geq \eps |x-x_0| \right\}.
\]
The main claim is then the following:
for every $\eps>0$, there exists $\delta>0$ such
that, for every $x_0 \in \Gamma_{\sfrac32}(u) \cap B_{\sfrac{s}{2}}$, 
\begin{gather}
u>0 \quad \text{on } \;C^+(x_0, \eps) \cap B_\delta(x_0),\label{e:claimconi_1}\\
u=0 \quad \text{on } \; C^-(x_0, \eps) \cap B_\delta(x_0). \label{e:claimconi_2}
\end{gather}
For what concerns \eqref{e:claimconi_1} 
assume by contradiction that there exist $x_j \in \Gamma_{\sfrac32}(u)
\cap B_{\sfrac{s}{2}}$
with $x_j \to x_0 \in \Gamma_{\sfrac32}(u) \cap \bar{B}_{\sfrac{s}{2}}$,
and $y_j \in C^+(x_j,\eps)$ with $y_j - x_j \to 0$ such that $u(y_j)=0$.

By the $C^{1,\sfrac12}$ regularity of the solution,
\eqref{e:unique trace} and \eqref{e:blowup vicini}, 
the rescalings $u^{x_j}_{r_j}$ with $r_j :=|y_j-x_j|$, 
converge uniformly to $u^{x_0}_0$. 
Up to subsequences, by the H\"older continuity of the normals
in \eqref{e:normali vicine} we can assume that  
$r_j^{-1}(y_j-x_j) \to z \in C^+(x_0,\eps)\cap\mathbb{S}^{n-1}$
and by uniform convergence (cp.~the Appendix~\ref{a:frequency})
$u^{x_0}_0(z) =0$.
This contradicts the fact that $x_0\in \Gamma_{\sfrac32}(u)$
and $u^{x_0}_0>0$ on $C^+(x_0,\eps)$.
Clearly, the proof of \eqref{e:claimconi_2} is at all 
analogous and the details are left to the readers.

\medskip

We can now easily conclude that
$\Gamma_{\sfrac32}(u) \cap B_{\sfrac{s}{2}}$ is the graph of a function $g$, 
for a suitably chosen small $\rho_1>0$. Without loss of generality assume that
$e(0) = e_{n-1}$ and set 
\[
g(x'):= \max\big\{t \in \R \,:\, (x',t,0) \in \Lambda(u)\big\}
\] 
for all points $x'\in \R^{n-2}$ with $|x'| \leq \delta \sqrt{1-\eps^2}$.
Note that by \eqref{e:claimconi_1} and \eqref{e:claimconi_2}
this maximum exists and belongs to $[-\eps\delta, \eps\delta]$.
Moreover $u(x',t,0) =0$ for every $-\eps\,\delta <t<g(x')$
and $u(x',t,0)>0$ for every $g(x')<t <\eps\,\delta$.
Eventually, by applying \eqref{e:claimconi_1} and \eqref{e:claimconi_2}
with respect to arbitrary $\eps$, we deduce that $g$ is differentiable
and in view of \eqref{e:normali vicine} we can conclude that 
$g$ is $C^{1,\alpha}$ regular for a suitable $\alpha>0$.
\end{proof}

\appendix

\section{}\label{a:frequency}
In this section we recall few known results concerning
the solutions to the Signorini problem, which are mainly contained
in the references \cite{Athan-Caff, Athan-Caff-Sal, Garofalo-Petrosyan, PSU}.

\subsection{Frequency function}
We recall the definition of frequency function:
for $x_0 \in \Gamma(u)$ and $0<r < 1 - |x_0|$
\[
N^{x_0}(r,u) := \frac{r\int_{B_r(x_0)}|\nabla u|^2\,dx}{\int_{\de B_r(x_0)} u^2\,d\cHn},
\]
if $u\vert_{\de B_r(x_0)} \not\equiv 0$ (note that $u\vert_{\de B_r(x_0)} \equiv 0$
if and only if $u\vert_{B_1} \equiv0$).

As proven in \cite{Athan-Caff-Sal} the function $(0, \dist(x_0,\de B_1))
\ni r\mapsto N^{x_0}(r, u)$ is nondecreasing for every $x_0 \in \Bn$.
The proof of this statement is an immediate consequence of \eqref{e:H'}, 
\eqref{e:D'} and \eqref{e:D}: without loss of generality we can assume
$x_0 =0$ and compute
\begin{align}\label{e:mon freq}
\frac{N'(r,u)}{N(r,u)} & = \frac{1}{r} + \frac{D'(r)}{D(r)} - \frac{H'(r)}{H(r)}
\notag\\
&= 2 \left( \frac{\int_{\de B_r} (\nabla u\cdot \nu)^2\,d\cHn}{\int_{\de B_r} u\nabla u\cdot \nu\,d\cHn} - 
\frac{\int_{\de B_r} u\nabla u\cdot \nu\,d\cHn}{\int_{\de B_r} u^2\,d\cHn} \right) \geq 0,
\end{align}
where the last inequality follows from the Cauchy--Schwartz inequality.
Since infimum of a sequence of decreasing functions, 
the map $x_0 \mapsto N^{x_0}(0^+,u)$ turns then out to be upper-semicontinuous.

Analyzing the case of equality in \eqref{e:mon freq} is important
for later applications: it is clear from the Cauchy--Schwartz inequality
that $N(r,u)\equiv \kappa$ for some $\kappa \in \R$
if and only if $u$ is $\kappa$-homogeneous, i.e.~$x\cdot \nabla u - \kappa \,u \equiv 0$.

As a consequence of the monotonicity of the frequency function
we also can prove the estimates \eqref{e:H dall'alto} and \eqref{e:H dal basso},
that we restate for readers' convenience.

\begin{lemma}\label{l:altezza}
Let $u \in H^1(B_1)$ be a solution to the Signorini problem
and assume that $0 \in \Gamma_{\sfrac32}(u)$. Then
the function $(0,1)\ni r\mapsto \frac{H(r)}{r^{n+2}}$ is nondecreasing
and in particular
\begin{equation}\label{e:H dall'alto_1}
H(r) \leq H(1)\, r^{n+2} \quad\forall \; 0<r<1,
\end{equation}
and for every $\eps>0$ there exists $r_0(\eps)>0$ such that
\begin{equation}\label{e:H dal basso_1}
H(r) \geq \frac{H(r_0)}{r_0^{n+2+\eps}}\, r^{n+2+\eps} \quad\forall \; 0<r<r_0.
\end{equation}
\end{lemma}

\begin{proof}
We start computing the following derivative:
\begin{equation}\label{e:log H}
\frac{d}{dr} \left(\log \frac{H(r)}{r^{n+2}} \right)  = 2\frac{D(r)}{H(r)} - \frac{3}{r}.
\end{equation}
As an immediate consequence of the monotonicity of the frequency $N(r,u) \geq N(0^+,u)=\sfrac32$
we deduce that $\frac{d}{dr} \left(\log \frac{H(r)}{r^{n+2}} \right) \geq 0$ from which \eqref{e:H dall'alto}
clearly follows.

Similarly, by the monotonicity of the frequency function,
for every $\eps>0$ there exists $r_0=r_0(\eps)>0$ such that 
\[
N(r,u) \leq N(0^+,u) +\frac{\eps}{2} = \frac{3+\eps}{2} \quad \forall\; 0<r< r_0.
\]
Therefore we infer from \eqref{e:log H} that
\[
\frac{d}{dr} \left(\log \frac{H(r)}{r^{n+2}} \right)  = \frac{2}{r} \left(N(r,u) - \frac{3}{2}\right) \leq 
\frac{\eps}{r} \quad \forall\; 0<r< r_0, 
\]
and integrating this differential inequality \eqref{e:H dal basso} follows at once. 
\end{proof}

\subsection{Optimal regularity}
Finally we recall the optimal regularity for the solution to the Signorini
problem proven in \cite{Athan-Caff}.

\begin{theorem}\label{t:optimal reg}
Let $u \in H^1(B_1)$ be a solution to the Signorini problem.
Then, $u \in C^{1,\sfrac12}(B_{\sfrac12}^+)$ and there exists a dimensional constant $C>0$
such that
\begin{equation}\label{e:optimal reg}
\|u\|_{C^{1,\sfrac12}(B_1^+)} \leq C\, \|u\|_{L^2(B_1)}.
\end{equation}
\end{theorem}

%
%

\bibliographystyle{plain}

\end{document}